\documentclass[hidelinks,onefignum,onetabnum]{siamart251216}

\usepackage{graphicx}
\usepackage{subcaption}
\usepackage{multirow}
\usepackage{booktabs}
\usepackage{makecell}
\usepackage{lipsum}
\usepackage{amsfonts}
\usepackage{graphicx}
\usepackage{epstopdf}
\usepackage{tabularx}
\usepackage{array}

\usepackage{algorithmic}
\ifpdf
  \DeclareGraphicsExtensions{.eps,.pdf,.png,.jpg}
\else
  \DeclareGraphicsExtensions{.eps}
\fi

\newsiamremark{remark}{Remark}
\newsiamremark{hypothesis}{Hypothesis}
\crefname{hypothesis}{Hypothesis}{Hypotheses}
\newsiamthm{claim}{Claim}
\newsiamremark{fact}{Fact}
\crefname{fact}{Fact}{Facts}

\headers{AlgMortar: a fully algebraic multiscale mortar preconditioner}{L. F. Santos, F. S. Sousa, R. F. Ausas, R. T. Guiraldello, F. Pereira}

\title{AlgMortar: a fully algebraic multiscale mortar preconditioner\thanks{
\funding{The authors acknowledge financial support from the National Agency of Petroleum, Natural Gas and Biofuels (ANP/Petrobras, Brazil) grant 2023/00510-0.
This work is partially supported by the National Science Foundation (NSF, USA) grant 2401945. F.S.S. and R.F.A. are partially funded by National Council for Scientific and Technological Development (CNPq, Brazil) grants 312372/2023-0 and 308704/2022-3, respectively.
L.F.S. and F.S.S. acknowledge support from the São Paulo Research Foundation (FAPESP, Brazil), grant 2026/08108-5.
This research was carried out using the computational resources of the Center for Mathematical Sciences Applied to Industry (CeMEAI) funded by FAPESP grant 2013/07375-0.
Any opinions, findings, and conclusions or recommendations expressed in this material are those of the authors and do not necessarily reflect the views of NSF, ANP/Petrobras, CNPq or FAPESP. 
}}}

\author{Luan F. Santos\thanks{Instituto de Ciências Matemáticas e de Computação,
Universidade de São Paulo, Av. Trab. São Carlense 400,
13566-590, São Carlos-SP, Brazil
(\email{luan.santos@icmc.usp.br}, \email{fsimeoni@icmc.usp.br},
\email{rfausas@icmc.usp.br}).}
\and Fabricio S. Sousa\footnotemark[2]
\and Roberto F. Ausas\footnotemark[2]
\and Rafael T. Guiraldello\thanks{Piri Technologies, LLC,
1000 E. University Ave., Dept. 4311,
Laramie, WY 82071-2000, USA
(\email{rafaeltrevisanuto@gmail.com}).}
\and Felipe Pereira\thanks{Department of Mathematical Sciences,
The University of Texas at Dallas, 800 W. Campbell Road,
Richardson, TX 75080-3021, USA
(\email{felipe.pereira@utdallas.edu}).}}

\usepackage{amsopn}

\begin{document}

\maketitle

\begin{abstract}
The solution of large-scale symmetric positive definite (SPD) linear systems arising from discretizations of second-order elliptic equations is challenging, especially in applications with highly heterogeneous coefficients, such as flow in porous media, which can lead to severely ill-conditioned systems. In this context, multiscale methods have recently been used to accelerate Krylov subspace methods, owing to their favorable parallel scalability. In this work, we present AlgMortar, a fully algebraic realization of the Multiscale Mortar Mixed Finite Element Method (MMMFEM). The method uses only information extracted from the fine-grid system matrix, which facilitates its implementation in existing solvers. AlgMortar uses graph partitioning to define a domain decomposition directly from the matrix graph. On each subdomain, it builds local linear systems that mimic Dirichlet problems, and couples the resulting local solutions through an algebraic interface condition that recovers the weak flux-continuity mechanism of MMMFEM. We prove that the method is well posed when the matrix is SPD. Numerical experiments on fine-grid linear systems arising from finite-volume discretizations of Darcy flow problems show that, when used as a preconditioner for the conjugate gradient method, the proposed approach exhibits good scalability and is competitive with state-of-the-art algebraic multigrid (AMG) methods for challenging heterogeneous, high-contrast test cases, including highly irregular corner-point grids. Additional tests with SPD matrices arising from problems other than finite-volume discretizations of Darcy flow showed competitive performance with AMG.
\end{abstract}



\section{Introduction}
In this paper, we consider the solution of sparse and typically ill-conditioned
symmetric positive definite (SPD) linear systems, such as those arising from fine-grid
discretizations of second-order elliptic problems of the form
\begin{equation}
\label{eq:elliptic-problem}
-\nabla\cdot\left(K\nabla p\right)=f
\quad \text{in } \Omega\subset\mathbb{R}^d,\qquad d=2,3,
\end{equation}
supplemented with suitable boundary conditions. Here, $p$ usually represents pressure,
$f$ is a source term, and $K$ is a SPD tensor.
Problems of this type arise in many applications, including flow in
porous media, subsurface reservoir simulation, groundwater flow, and heat
conduction in heterogeneous materials. Their accurate numerical approximation
often requires fine grids, leading to large-scale linear systems.

To leverage massively parallel architectures, linear solvers must exhibit good
parallel scalability. In particular, scalable preconditioners are essential for
Krylov subspace methods such as the preconditioned conjugate gradient (PCG)
method and GMRES.
Domain decomposition preconditioners \cite{toselli2005dd,dolean2015ddm} are well suited to this task because they decompose the global problem into local subproblems that can be solved concurrently.
Multiscale methods form an important class of domain decomposition approaches.
They construct basis functions that capture local heterogeneity and couple them
through a coarse problem whose solution defines the global approximation. This strategy underlies a broad range of multiscale
frameworks
\cite{hou1997mfem,jenny2003multiscalefv,efendiev2013gmsfem,lipnikov2008m3,moyner2016msrsb}.


Several multiscale methods based on the mixed formulation of
Eq.~\eqref{eq:elliptic-problem} have also been proposed. The Multiscale Mortar
Mixed Finite Element Method (MMMFEM) \cite{arbogast2007mortar} constructs
multiscale basis functions from local Dirichlet problems and weakly enforces
continuity of the normal flux. In contrast, the Multiscale Hybrid-Mixed method
(MHM) \cite{harder2013mhm} uses local Neumann problems and weakly imposes
pressure continuity across interfaces. The Multiscale Robin Coupled Method
(MRCM) \cite{guiraldello2018mrcm} uses Robin boundary conditions and weakly
enforces continuity of both pressure and normal flux. For suitable asymptotic limits of the Robin parameter, MRCM recovers either MMMFEM or MHM.

The use of multiscale methods as preconditioners has been explored in several works, including
\cite{ye2024twoleveloverlap,fu2024preconditioner,changqing2025highconstrast,carvalho2025preconditioner,jaramillo2025porenetwork}.
In particular, \cite{carvalho2025preconditioner} employs MRCM as a PCG preconditioner for a global pressure system arising from a fine-grid finite-volume discretization of Eq.~\eqref{eq:elliptic-problem}. The reported results show that the best computational performance is achieved in the MMMFEM limit, with good scalability and competitive performance relative to state-of-the-art algebraic multigrid (AMG) methods.


The multiscale methods described above typically rely on geometric information, PDE coefficients, and discretization-specific details. In industrial subsurface-flow simulations, however, such information may be highly complex, proprietary, unavailable, or costly to transfer between codes. Algebraic preconditioners address this limitation by operating solely on the global discretization matrix and right-hand side (RHS), allowing them to accommodate arbitrary grids, unstructured discretizations, mixed-dimensional models, and highly heterogeneous coefficients without requiring the manual reconstruction of coarse meshes or geometric neighborhoods.
This operator-only approach simplifies integration with legacy reservoir simulators, streamlines preprocessing and model handoff between teams, and reduces sources of error and overhead associated with grid conversion. Moreover, algebraic frameworks enable robust, black-box multilevel coarsening and interpolation strategies that preserve important algebraic features (such as connectivity and anisotropy) of the discrete operator, improving convergence in challenging heterogeneous and anisotropic media and supporting scalable parallel implementations required for large industrial problems. 

For these reasons, several algebraic multiscale approaches have been proposed to construct coarse spaces and multiscale operators directly from the algebraic system. These include the two-stage algebraic multiscale (TAMS) method \cite{zhou2011tams}, the algebraic multiscale solver (AMS) \cite{wang2014algebraic}, the multiscale two-point flux approximation method (MsTPFA) \cite{moyner2014tpfa}, and other related formulations \cite{tene2016fams,akbari2020robinalgebraic,desouza2022algebraic,zhou2026algebraicmultiscalepreconditionerlarge}.

Motivated by the MRCM preconditioning results in the MMMFEM limit \cite{carvalho2025preconditioner} and by the need for algebraic solvers, we introduce in this paper AlgMortar, a fully algebraic realization of MMMFEM. Using only an SPD matrix, AlgMortar constructs graph-based subdomains, local algebraic Dirichlet operators, and interface conditions that reproduce the weak flux-continuity condition. It exactly recovers MMMFEM on uniform Cartesian grids when $K$ is the identity tensor and remains effective as a PCG preconditioner for heterogeneous problems.

We establish the well-posedness of AlgMortar for SPD matrices. Both the interface system and the reduced  system obtained by eliminating the interface variable are positive definite, enabling the use of AlgMortar as a PCG preconditioner. We also derive condition-number estimates showing that spectral clustering depends on how well the interface space approximates the signed interface quantities on graph-partition interfaces.

Finally, we assess AlgMortar as a PCG preconditioner for fine-grid finite-volume systems arising from Eq.~\eqref{eq:elliptic-problem} on Cartesian and irregular corner-point grids. Numerical experiments demonstrate favorable parallel scalability and show that AlgMortar can outperform AMG on large-scale problems with high-contrast permeability fields. Additional SPD test problems arising from other applications and discretizations also indicate that AlgMortar can achieve performance competitive with AMG in some cases.

The paper is organized as follows. Section~\ref{sec:mmmfem} reviews MMMFEM, and Section~\ref{sec:alg-mortar} introduces AlgMortar. Its well-posedness, use as a PCG preconditioner, and implementation details are presented in Sections~\ref{sec:well-posed}-\ref{sec:mortar-pc-algorithm}. Section~\ref{sec:num-exp} reports the numerical results, while Section~\ref{sec:conclusions} presents our concluding remarks.

\section{MMMFEM overview}
\label{sec:mmmfem}

This section reviews the multiscale mixed mortar finite element method (MMMFEM) \cite{arbogast2000mortar,arbogast2007mortar}. The method combines domain decomposition with local mixed finite element discretizations and introduces the interface pressure as an additional unknown, called the mortar variable. We present its formulation and solution using multiscale basis functions and a global interface problem, emphasizing the components later reproduced algebraically. For brevity, we consider only Dirichlet boundary conditions; more general cases are discussed in \cite{ganis2009multiflux,guiraldello2018mrcm}.

\subsection{MMMFEM variational formulation}
\label{sec:intro-mmmfem-var}

Let $\Omega\subset\mathbb{R}^d$, $d=2,3$, be a bounded Lipschitz domain. Let
$K\in[L^\infty(\Omega)]^{d\times d}$ be symmetric and uniformly positive definite tensor,
$f\in L^2(\Omega)$ the source term, and
$g\in H^{1/2}(\partial\Omega)$ the Dirichlet boundary data.

We assume that $\Omega$ can be partitioned into $m$ non-overlapping blocks. For $i\neq j$, define the interface
$\Gamma_{ij}=\partial\Omega_i\cap\partial\Omega_j$, the skeleton
$\Gamma=\bigcup_{1\leq i<j\leq m}\Gamma_{ij}$, and the local skeleton
$\Gamma_i=\partial\Omega_i\cap\Gamma$.
On each $\Omega_i$, let $\mathcal{T}_h^i$ be a local mesh of size $h$, and consider mixed finite element spaces
$\boldsymbol{V}_{h,i}\subset H(\operatorname{div};\Omega_i)$ and
$Q_{h,i}\subset L^2(\Omega_i)$, such as the Raviart-Thomas
 spaces \cite{raviart1977mixed}. Let $\boldsymbol{n}_i$ denote the outward unit normal to $\partial\Omega_i$, and let $M_H\subset L^2(\Gamma)$ be the mortar space used to approximate the interface pressure.

The MMMFEM variational formulation is to find
$(\boldsymbol{u}_{h,i},p_{h,i})\in
\boldsymbol{V}_{h,i}\times Q_{h,i}$, $i=1,\ldots,m$, and
$\lambda_H\in M_H$ such that
\begin{align}
(K^{-1}\boldsymbol{u}_{h,i},\boldsymbol{v})_{\Omega_i}
-
(p_{h,i},\nabla\cdot\boldsymbol{v})_{\Omega_i}
&=
-\langle\lambda_H,\boldsymbol{v}\cdot\boldsymbol{n}_i\rangle_{\Gamma_i}
-\langle g,\boldsymbol{v}\cdot\boldsymbol{n}_i\rangle_{\partial\Omega_i\cap\partial\Omega},
\label{eq:mmmfem-local-darcy}
\\
(\nabla\cdot\boldsymbol{u}_{h,i},q)_{\Omega_i}
&=
(f,q)_{\Omega_i},
\label{eq:mmmfem-local-mass}
\\
\sum_{i=1}^{m}
\langle\boldsymbol{u}_{h,i}\cdot\boldsymbol{n}_i,\mu\rangle_{\Gamma_i}
&=
0,
\label{eq:mmmfem-flux-continuity}
\end{align}
for all $\boldsymbol{v}\in\boldsymbol{V}_{h,i}$,
$q\in Q_{h,i}$, and $\mu\in M_H$.

For a given $\lambda_H$, Eqs.~\eqref{eq:mmmfem-local-darcy} and~\eqref{eq:mmmfem-local-mass} define independent mixed problems on each subdomain, with pressure data $\lambda_H$ on $\Gamma_i$ and $g$ on $\partial\Omega_i\cap\partial\Omega$. Equation~\eqref{eq:mmmfem-flux-continuity} couples the subdomains by weakly enforcing continuity of the normal flux across interfaces.

\subsection{Multiscale basis functions and the interface problem}
\label{sec:intro-mbf}

Eqs.~\eqref{eq:mmmfem-local-darcy}-\eqref{eq:mmmfem-flux-continuity} can be solved efficiently and in parallel using multiscale basis functions (MBFs). Let $\{\psi_k\}_{k=1}^{N_H}$ be a basis of the mortar space $M_H$, and write
$\lambda_H=\sum_{k=1}^{N_H}\lambda_k\psi_k$.

On each subdomain $\Omega_i$, we first compute the particular solution
$(\overline{\boldsymbol{u}}_{h,i},\overline{p}_{h,i})
\in\boldsymbol{V}_{h,i}\times Q_{h,i}$ associated with $f$ and $g$, with zero pressure on $\Gamma_i$:
\begin{align}
(K^{-1}\overline{\boldsymbol{u}}_{h,i},\boldsymbol{v})_{\Omega_i}
-
(\overline{p}_{h,i},\nabla\cdot\boldsymbol{v})_{\Omega_i}
&=
-\langle g,\boldsymbol{v}\cdot\boldsymbol{n}_i\rangle_{\partial\Omega_i\cap\partial\Omega},
&& \forall \boldsymbol{v}\in\boldsymbol{V}_{h,i},
\label{eq:mbf-particular-darcy}
\\
(\nabla\cdot\overline{\boldsymbol{u}}_{h,i},q)_{\Omega_i}
&=
(f,q)_{\Omega_i},
&& \forall q\in Q_{h,i}.
\label{eq:mbf-particular-mass}
\end{align}

For each $\psi_k$, the corresponding MBF on $\Omega_i$ is
$(\widehat{\boldsymbol{u}}_{h,i}^{(k)},\widehat{p}_{h,i}^{(k)})
\in\boldsymbol{V}_{h,i}\times Q_{h,i}$ satisfying
\begin{align}
(K^{-1}\widehat{\boldsymbol{u}}_{h,i}^{(k)},\boldsymbol{v})_{\Omega_i}
-
(\widehat{p}_{h,i}^{(k)},\nabla\cdot\boldsymbol{v})_{\Omega_i}
&=
-\langle\psi_k,\boldsymbol{v}\cdot\boldsymbol{n}_i\rangle_{\Gamma_i},
&& \forall \boldsymbol{v}\in\boldsymbol{V}_{h,i},
\label{eq:mbf-basis-darcy}
\\
(\nabla\cdot\widehat{\boldsymbol{u}}_{h,i}^{(k)},q)_{\Omega_i}
&=
0,
&& \forall q\in Q_{h,i}.
\label{eq:mbf-basis-mass}
\end{align}
Thus, each MBF solves a homogeneous local problem with $\psi_k$ as the only nonzero pressure datum on the interface.
By linearity,
\begin{equation}
\label{eq:local-solution-decomposition}
\boldsymbol{u}_{h,i}
=
\overline{\boldsymbol{u}}_{h,i}
+
\sum_{k=1}^{N_H}\lambda_k
\widehat{\boldsymbol{u}}_{h,i}^{(k)},
\qquad
p_{h,i}
=
\overline{p}_{h,i}
+
\sum_{k=1}^{N_H}\lambda_k
\widehat{p}_{h,i}^{(k)}.
\end{equation}

After computing the MBFs and particular solutions, the only global unknowns are the mortar coefficients $\{\lambda_k\}_{k=1}^{N_H}$. Substituting Eq.~\eqref{eq:local-solution-decomposition} into Eq.~\eqref{eq:mmmfem-flux-continuity} yields
\begin{align}
\sum_{k=1}^{N_H}\lambda_k
\sum_{i=1}^{m}
\left\langle
\widehat{\boldsymbol{u}}_{h,i}^{(k)}\cdot\boldsymbol{n}_i,
\mu
\right\rangle_{\Gamma_i}
&=
-
\sum_{i=1}^{m}
\left\langle
\overline{\boldsymbol{u}}_{h,i}\cdot\boldsymbol{n}_i,
\mu
\right\rangle_{\Gamma_i},
&& \forall \mu\in M_H.
\label{eq:global-mortar-mbf-system}
\end{align}

The interface matrix is SPD, provided that the mortar space is not excessively rich relative to the normal-flux trace space \cite{arbogast2000mortar}. After solving the interface problem, the local pressure and velocity fields are reconstructed independently using Eq.~\eqref{eq:local-solution-decomposition}.

\subsection{Local lowest-order Raviart-Thomas solvers}
\label{sec:intro-rt0}

Although MMMFEM allows different local mixed finite element spaces and nonmatching grids, we focus on matching grids and lowest-order Raviart-Thomas spaces, which are widely used in practice \cite{arbogast1997cellcentered}. This is the setting reproduced algebraically in Section~\ref{sec:alg-mortar}.

Let $\mathcal{E}_h$ denote the set of fine-grid interface facets on $\Gamma$. The mortar space $M_H$ is taken as a subspace of
\begin{equation*}
F(\mathcal{E}_h)
=
\left\{
\lambda:\Gamma\to\mathbb{R}
\;\middle|\;
\lambda|_e\in\mathbb{P}_0(e),
\quad \forall e\in\mathcal{E}_h
\right\}.
\end{equation*}

On each subdomain $\Omega_i$, $Q_{h,i}$ denotes the piecewise constant pressure space and $\boldsymbol{V}_{h,i}$ the lowest-order Raviart-Thomas velocity space; see \cite{guiraldello2018mrcm} for details. When $K$ is diagonal, these local mixed problems are equivalent, up to quadrature error, to a cell-centered finite-volume discretization using the two-point flux approximation (TPFA) \cite{arbogast1997cellcentered,guiraldello2018mrcm}. The fluxes are recovered from the cell pressures by finite differences \cite{guiraldello2018mrcm} and used to assemble the interface system~\eqref{eq:global-mortar-mbf-system}.

\section{The algebraic multiscale mortar method (AlgMortar)}
\label{sec:alg-mortar}

We consider the linear system
\begin{equation}
\label{eq:linear-system}
Ax=b,
\end{equation}
where, hereafter, $A\in\mathbb{R}^{N\times N}$ denotes an SPD matrix, $x\in\mathbb{R}^N$ the unknown vector, and $b\in\mathbb{R}^N$ the RHS. In the context of Darcy flow problems, $x$ represents the vector of discrete pressures.
Using only $A$ and $b$, without access to the grid, discretization, or PDE coefficients, we now describe the construction of AlgMortar.
Our construction ensures that, when $A$ arises from a TPFA discretization on a uniform Cartesian grid with homogeneous permeability $K=I$, where $I$ denotes the identity tensor, AlgMortar recovers MMMFEM up to a trivial rescaling of the mortar variable.

\subsection{Graph partition}
\label{sec:graph-partition}

We represent the computational domain by the index set $\Omega=\{1,\ldots,N\}$, whose elements correspond to the unknowns of Eq.~\eqref{eq:linear-system}. From $A$, we define the adjacency graph $G(A)=(\Omega,E)$, where $(k,l)\in E$ if $k \neq l$ and $a_{kl}\neq0$. The graph is then partitioned into $m$ non-overlapping subdomains, $\Omega=\bigcup_{i=1}^m\Omega_i$, using a graph partitioner such as METIS \cite{karypis1997metis}. This yields an algebraic domain decomposition based only on the connectivity pattern of $A$, as commonly used in domain decomposition \cite{dolean2015ddm,vecharynski2014partitioning} and algebraic multiscale methods \cite{moyner2014tpfa,jaramillo2025porenetwork,zhou2026algebraicmultiscalepreconditionerlarge,liu2026graphbased}.

For $i\neq j$, we write $i\sim j$ when $\Omega_i$ and $\Omega_j$ share an interface, that is, when $a_{kl}\neq0$ for some $k\in\Omega_i$ and $l\in\Omega_j$. We then reorder the unknowns by subdomain, placing those associated with $\Omega_1,\Omega_2,\ldots,\Omega_m$ consecutively, and let $P$ denote the corresponding permutation matrix.
The permuted system is
\begin{equation}
\label{eq:linear-system-perm}
\widetilde{A}\widetilde{x}=\widetilde{b},
\qquad
\widetilde{A}=PAP^T,\quad
\widetilde{x}=Px,\quad
\widetilde{b}=Pb.
\end{equation}

Note that permutation preserves symmetry and positive definiteness. For
simplicity, we henceforth drop the tildes and work with the permuted
quantities. 

For each pair $i \sim j$, we define the set of
interface edges by
\begin{equation*}
\label{eq:Eij}
\mathcal{E}_{ij}
=
\left\{
(k,l)
\;:\;
k \in \Omega_i,\;
l \in \Omega_j,\;
a_{kl} \neq 0
\right\}.
\end{equation*}

Since $A$ is symmetric, $(k,l)\in\mathcal{E}_{ij}$ if and only if $(l,k)\in\mathcal{E}_{ji}$. We identify them as the same edge and order $\mathcal{E}_{ij}$ and $\mathcal{E}_{ji}$ consistently.
We define the local and global interface edge sets, forming the algebraic skeleton, by
\begin{equation}
\label{eq:E}
\mathcal{E}_i
=
\bigcup_{i\sim j}\mathcal{E}_{ij},
\qquad
\mathcal{E}
=
\bigcup_{\substack{i<j\\ i\sim j}}\mathcal{E}_{ij}.
\end{equation}

\subsection{Notation}

We introduce some notation used throughout the algebraic formulation.
For finite index sets $X$ and $Y$, let $\mathbb{R}^X$ denote the space
of real vectors indexed by $X$, and let $\mathbb{R}^{X \times Y}$
denote the space of real matrices with rows indexed by $X$ and columns
indexed by $Y$.

Any $x \in \mathbb{R}^{\Omega}$ can be written as
\begin{equation*}
    x = \left(x_{\Omega_1}, \ldots, x_{\Omega_m}\right)^T,
    \qquad
    x_{\Omega_i} \in \mathbb{R}^{\Omega_i},
\end{equation*}
where $x_{\Omega_i}=R_{\Omega_i}x$ and $R_{\Omega_i}\in\mathbb{R}^{\Omega_i\times\Omega}$ is the restriction operator from $\mathbb{R}^{\Omega}$ to $\mathbb{R}^{\Omega_i}$.

For a matrix $A \in \mathbb{R}^{\Omega \times \Omega}$ and any $i,j=1,\ldots,m$, we define
$A_{ij}:=R_{\Omega_i}AR_{\Omega_j}^{T}\in\mathbb{R}^{\Omega_i\times\Omega_j}$.
Thus, for $i\neq j$, $A_{ij}\neq0$ if and only if $i\sim j$.

\subsection{Local equation and diagonal modification}

Restricting the global system (Eq.~\eqref{eq:linear-system}) to the subdomain
$\Omega_i$ yields
\begin{equation}
\label{eq:local-original}
A_{ii}x_{\Omega_i}
+
\sum_{j:\,j\sim i}A_{ij}x_{\Omega_j}
=
b_{\Omega_i}.
\end{equation}

For each pair of neighboring subdomains $i\sim j$, we define the
diagonal matrix
\begin{equation}
\label{eq:Dabs}
D_{ij}
:=
\operatorname{diag}
\left(
|A_{ij}|\boldsymbol{1}_{\Omega_j}
\right),
\end{equation}
whose diagonal entries are the absolute row sums of the coupling block
$A_{ij}$. Here, $|A_{ij}|$ denotes the entrywise absolute value of
$A_{ij}$, and $\boldsymbol{1}_{\Omega_j}$ denotes the vector of ones
indexed by $\Omega_j$. Thus, for each $k\in\Omega_i$,
\begin{equation}
\label{eq:Dij-component}
(D_{ij})_{kk}
=
\sum_{l:\,(k,l)\in\mathcal{E}_{ij}} |a_{kl}|.
\end{equation}

We define the modified local operator by
\begin{equation}
\label{eq:Aiabs}
A_{\Omega_i}
:=
A_{ii}
+
\sum_{j:\,j\sim i}
D_{ij}.
\end{equation}

Then, Eq.~\eqref{eq:local-original} can be equivalently rewritten as
\begin{equation}
\label{eq:local-modified}
A_{\Omega_i}x_{\Omega_i}
=
b_{\Omega_i}
-
\sum_{j:\,j\sim i}
\left(
A_{ij}x_{\Omega_j}
-
D_{ij}x_{\Omega_i}
\right).
\end{equation}

Notice that Eq.~\eqref{eq:local-modified} is only an algebraic
rearrangement of the original local system. Moreover,
$A_{\Omega_i}$ is SPD whenever $A$ is SPD.

\begin{remark}
\label{rem:tpfa-local-dirichlet}
If $A$ arises from a global TPFA discretization on a uniform Cartesian
grid with $K=I$, the modification in Eq.~\eqref{eq:Aiabs} makes
$A_{\Omega_i}$ coincide with the TPFA operator for the local Dirichlet
problem on $\Omega_i$. Similar diagonal modifications have been used in
other works to impose physical boundary conditions
\cite{moyner2014tpfa,akbari2020robinalgebraic,zhou2026algebraicmultiscalepreconditionerlarge,liu2026graphbased}.
\end{remark}

We then define the block-diagonal matrix
\begin{equation}
\label{eq:Ahat}
\widehat{A}
:=
\operatorname{diag}
\left(
A_{\Omega_1},\ldots,A_{\Omega_m}
\right).
\end{equation}
Since each block $A_{\Omega_i}$ is SPD whenever $A$ is SPD, it follows
that $\widehat{A}$ is also SPD.

\subsection{Decomposing the global problem into local problems}

We define the signed interface operator
$M\in\mathbb{R}^{\mathcal{E} \times \Omega}$
edgewise by
\begin{equation}
\label{eq:M-edgewise}
(Mx){(k,l)}
:=
x_k-\operatorname{sign}(a_{kl})x_l,
\qquad
(k,l)\in\mathcal{E}.
\end{equation}

In particular, if the off-diagonal entries of $A$ are negative, then $M$ represents the sum of the unknowns
associated with the two endpoints of each interface edge.

We also define the diagonal weight matrix
$W\in\mathbb{R}^{\mathcal{E}\times\mathcal{E}}$ by
\begin{equation}
\label{eq:W}
[W]_{(k,l),(k,l)}
=
|a_{kl}|,
\qquad (k,l)\in\mathcal{E}.
\end{equation}

Since $a_{kl}\neq0$ for every $(k,l)\in\mathcal{E}$, the matrix $W$ is positive definite.

Fix $i\in\{1,\ldots,m\}$. For each $j\sim i$, we define the matrix
$B_{ij}\in\mathbb{R}^{\Omega_i\times\mathcal{E}}$ associated with the
interface between $\Omega_i$ and $\Omega_j$. Its action on
$\mu\in\mathbb{R}^{\mathcal{E}}$ is given componentwise, for
$k\in\Omega_i$, by
\begin{equation}
\label{eq:Bij}
(B_{ij}\mu)_k
:=
\begin{cases}
-\displaystyle\sum_{l:\,(k,l)\in\mathcal{E}_{ij}}
|a_{kl}|\,\mu{(k,l)},
& i<j,\\[4mm]
\displaystyle\sum_{l:\,(k,l)\in\mathcal{E}_{ij}}
a_{kl}\,\mu{(l,k)},
& i>j.
\end{cases}
\end{equation}

The total interface matrix associated with $\Omega_i$ is then defined by
\begin{equation}
\label{eq:Bi}
B_i:=\sum_{j:\,j\sim i}B_{ij}.
\end{equation}

\begin{proposition}
\label{prop:local-mixed-form}
Equation~\eqref{eq:local-modified} can be written as
\begin{equation}
\label{eq:local-interface-form}
A_{\Omega_i}x_{\Omega_i}
=
b_{\Omega_i}
-
B_iMx,
\qquad
i=1,\ldots,m.
\end{equation}
\end{proposition}

\begin{proof}
Fix $i\in\{1,\ldots,m\}$ and let $j\sim i$. We first show that
\begin{equation}
\label{eq:Bij-identity}
B_{ij}Mx
=
A_{ij}x_{\Omega_j}-D_{ij}x_{\Omega_i}.
\end{equation}

Let $k\in\Omega_i$. By the definition of the matrix-vector product and
Eq.~\eqref{eq:Dij-component}, we have
\begin{equation}
\label{eq:Aij-Dij-component}
\left(A_{ij}x_{\Omega_j}-D_{ij}x_{\Omega_i}\right)_k
=
\sum_{l:\,(k,l)\in\mathcal{E}_{ij}}
\left(a_{kl}x_l-|a_{kl}|x_k\right).
\end{equation}

Suppose first that $i<j$. From Eqs.~\eqref{eq:M-edgewise}
and~\eqref{eq:Bij},
\begin{align*}
(B_{ij}Mx)_k
&=
-\sum_{l:\,(k,l)\in\mathcal{E}_{ij}}
|a_{kl}|(Mx)(k,l)
=
-\sum_{l:\,(k,l)\in\mathcal{E}_{ij}}
|a_{kl}|\left(x_k-\operatorname{sign}(a_{kl})x_l\right)
\\
&=
\sum_{l:\,(k,l)\in\mathcal{E}_{ij}}
\left(a_{kl}x_l-|a_{kl}|x_k\right).
\end{align*}

Comparing with Eq.~\eqref{eq:Aij-Dij-component} proves
Eq.~\eqref{eq:Bij-identity} for $i<j$.

Suppose now that $i>j$. For $(k,l)\in\mathcal{E}_{ij}$, the
corresponding globally oriented edge is
$(l,k)\in\mathcal{E}_{ji}\subset\mathcal{E}$. Therefore,
Eq.~\eqref{eq:M-edgewise} gives
\begin{equation*}
(Mx)(l,k)
=
x_l-\operatorname{sign}(a_{lk})x_k.
\end{equation*}

Using Eq.~\eqref{eq:Bij}, together with $a_{lk}=a_{kl}$, we obtain
\begin{align*}
(B_{ij}Mx)_k
&=
\sum_{l:\,(k,l)\in\mathcal{E}_{ij}}
a_{kl}(Mx)(l,k)
=
\sum_{l:\,(k,l)\in\mathcal{E}_{ij}}
a_{kl}\left(x_l-\operatorname{sign}(a_{lk})x_k\right)
\\
&=
\sum_{l:\,(k,l)\in\mathcal{E}_{ij}}
\left(a_{kl}x_l-|a_{kl}|x_k\right).
\end{align*}

By Eq.~\eqref{eq:Aij-Dij-component}, this proves
Eq.~\eqref{eq:Bij-identity} for $i>j$.

Finally, summing Eq.~\eqref{eq:Bij-identity} over all $j\sim i$ and
using Eq.~\eqref{eq:Bi}, we obtain
\begin{equation*}
B_iMx
=
\sum_{j:\,j\sim i}
\left(A_{ij}x_{\Omega_j}-D_{ij}x_{\Omega_i}\right).
\end{equation*}

Substituting this identity into Eq.~\eqref{eq:local-modified} yields
Eq.~\eqref{eq:local-interface-form}.
\end{proof}

We now introduce the interface variable
$\lambda:=Mx\in\mathbb{R}^{\mathcal{E}}$, which contains one degree of freedom for each globally oriented interface edge, and define the matrix
\begin{equation}
\label{eq:B}
B
:=
\begin{bmatrix}
B_1^T & \cdots & B_m^T
\end{bmatrix}^T
\in\mathbb{R}^{\Omega\times\mathcal{E}}.
\end{equation}

Using Proposition~\ref{prop:local-mixed-form}, the original problem can
then be decomposed into local problems coupled through the interface
variable $\lambda$:
\begin{equation}
\label{eq:exact-mixed-global}
\begin{aligned}
\widehat{A}x+B\lambda &= b,\\
\lambda-Mx &= 0,
\end{aligned}
\end{equation}
where $\widehat{A}$ is defined in Eq.~\eqref{eq:Ahat}.

Under the assumption that $A$ is SPD, eliminating
$\lambda=Mx$ recovers the original system~\eqref{eq:linear-system}.
Hence, no approximation has been introduced, and
Eq.~\eqref{eq:exact-mixed-global} is an exact algebraic reformulation
of $Ax=b$.

\begin{proposition}
\label{prop:BWM}
Let $M$, $W$, and $B$ be defined by
Eqs.~\eqref{eq:M-edgewise}, \eqref{eq:W}, and \eqref{eq:B}, respectively.
Then,
\begin{equation}
\label{eq:BWM}
B^T=-WM.
\end{equation}
\end{proposition}

\begin{proof}
Let $(k,l)\in\mathcal{E}_{ij}$, with $i<j$. By the definitions of
$B_{ij}$, $B_{ji}$, and $B$ in Eqs.~\eqref{eq:Bij} and~\eqref{eq:B},
together with $a_{lk}=a_{kl}$, the column of $B$
associated with $(k,l)$ has exactly two nonzero entries:
$-|a_{kl}|$ in the row corresponding to $k$ and $a_{kl}$ in the row
corresponding to $l$. Therefore,
\begin{equation*}
(B^Tx)(k,l)
=
-|a_{kl}|x_k+a_{kl}x_l
=
-|a_{kl}|
\left(
x_k-\operatorname{sign}(a_{kl})x_l
\right).
\end{equation*}

Using Eqs.~\eqref{eq:M-edgewise} and~\eqref{eq:W}, the right-hand side is precisely $-(WMx)(k,l)$. Since this holds for every $(k,l)\in\mathcal{E}$, we obtain Eq.~\eqref{eq:BWM}.
\end{proof}

\subsection{The AlgMortar method}

Let $\Lambda\subset\mathbb{R}^{\mathcal{E}}$ be a subspace, referred to
as the interface or mortar space. Rather than enforcing the exact
interface relation $\lambda=Mx$ on every interface edge, AlgMortar
enforces this relation weakly with respect to the inner product induced
by $W$. The AlgMortar formulation then consists in finding
$x\in\mathbb{R}^{\Omega}$ and $\lambda\in\Lambda$ such that
\begin{equation}
\label{eq:weak-mixed-global}
\begin{aligned}
\widehat{A}x+B\lambda&=b,\\
\mu^T W(\lambda-Mx)&=0,
\qquad \text{for every }\mu\in\Lambda.
\end{aligned}
\end{equation}

\begin{remark}
\label{rem:weak-tpfa}
If $A$ arises from a global TPFA discretization, then $a_{kl}<0$ for
every $(k,l)\in\mathcal{E}$. Hence,
$(Mx)(k,l)=x_k+x_l$ and
$W_{(k,l),(k,l)}=-a_{kl}$, where $-a_{kl}$ is the transmissibility
associated with $(k,l)$. For a TPFA discretization on a uniform
Cartesian grid with $K=I$, $\lambda$ can be identified with twice the
interface pressure. With this normalization, the second equation in
Eq.~\eqref{eq:weak-mixed-global} is equivalent to the MMMFEM weak
flux-continuity condition.
\end{remark}

\begin{remark}
\label{rem:mmmfem-recovery}
Remarks~\ref{rem:tpfa-local-dirichlet} and~\ref{rem:weak-tpfa} show
that, when $A$ arises from a global TPFA discretization on a uniform
Cartesian grid with $K=I$, the present AlgMortar formulation recovers
the MMMFEM formulation up to a factor of two in the normalization of
the mortar variable.
\end{remark}

\subsection{Multiscale basis functions and the interface problem}
\label{sec:alg-mortar-interface-problem}

Before establishing well-posedness, we follow
Subsection~\ref{sec:intro-mbf} and solve the AlgMortar formulation
using MBFs. For each subdomain $\Omega_i$, the particular solution is
defined by
\begin{equation}
\label{eq:def-xbar-local}
A_{\Omega_i}\overline{x}_{\Omega_i}
=
b_{\Omega_i}.
\end{equation}

Let $\{\psi_s\}_{s=1}^{N_{\Lambda}}$ be a basis of $\Lambda$, and define
\begin{equation*}
\mathcal{I}_i
:=
\left\{
s\in\{1,\ldots,N_{\Lambda}\}:
\operatorname{supp}(\psi_s)\cap\mathcal{E}_i\neq\emptyset
\right\}.
\end{equation*}
For each $s\in\mathcal{I}_i$, the corresponding local MBF is computed
from
\begin{equation}
\label{eq:def-xhat-local}
A_{\Omega_i}\widehat{x}_{\Omega_i}^{(s)}
=
-B_i\psi_s.
\end{equation}
For $s\notin\mathcal{I}_i$, we set $\widehat{x}_{\Omega_i}^{(s)}=0$, and define the corresponding global MBF as $\widehat{x}^{(s)}=\left(\widehat{x}_{\Omega_1}^{(s)},\ldots,\widehat{x}_{\Omega_m}^{(s)}\right)^T$.

Writing
$\lambda=\sum_{s=1}^{N_{\Lambda}}\lambda_s\psi_s$, the solution of the
first equation in Eq.~\eqref{eq:weak-mixed-global} can be decomposed as
$x=\overline{x}+\widehat{x}(\lambda)$, where
\begin{equation}
\label{eq:xhat-expansion}
\widehat{x}(\lambda)
=
\sum_{s=1}^{N_{\Lambda}}
\lambda_s\widehat{x}^{(s)}
=
-\widehat{A}^{-1}B\lambda.
\end{equation}

Substituting this decomposition into the second equation in
Eq.~\eqref{eq:weak-mixed-global} gives the interface problem
\begin{equation}
\label{eq:interface-problem-alg}
\mu^T
W\left(I+M\widehat{A}^{-1}B\right)\lambda
=
\mu^TWM\overline{x},
\qquad
\forall\,\mu\in\Lambda.
\end{equation}

Using Proposition~\ref{prop:BWM}, Eq.~\eqref{eq:interface-problem-alg}
can equivalently be written as
\begin{equation}
\label{eq:interface-problem-alg-sym}
\mu^T
\left(
W-B^T\widehat{A}^{-1}B
\right)\lambda
=
-\mu^TB^T\overline{x},
\qquad
\forall\,\mu\in\Lambda.
\end{equation}

The particular solutions and MBFs are computed independently on each
subdomain. Thus, each $A_{\Omega_i}$ can be factorized in parallel, and
the same factorization can be reused for all local solves. The MBFs are
then used to assemble and solve the interface problem
\eqref{eq:interface-problem-alg}, after which the global solution is
reconstructed as $x=\overline{x}+\widehat{x}(\lambda)$.

We typically assume that each $\psi_s$ is supported on a single
interface $\mathcal{E}_{ij}$. Consequently, each subdomain computes only
the MBFs associated with its neighboring interfaces, reducing the local
computational cost and yielding a sparse interface system.

\section{Well-posedness analysis of the AlgMortar method}
\label{sec:well-posed}

In this section, we establish the existence and uniqueness of the
AlgMortar solution. We first rewrite the AlgMortar formulation
\eqref{eq:weak-mixed-global} as a symmetric linear system and then show
that its matrix is positive definite. Throughout this section, $A$ is
assumed to be SPD.

\begin{proposition}
\label{prop:simetry}
Let $\Lambda\subset\mathbb{R}^{\mathcal{E}}$ be an interface space with
$\dim(\Lambda)=r$, and let
$Z\in\mathbb{R}^{\mathcal{E}\times r}$ have linearly independent
columns such that $\operatorname{range}(Z)=\Lambda$. Then, writing
$\lambda=Z\eta$, the AlgMortar formulation
\eqref{eq:weak-mixed-global} is equivalent to
\begin{equation}
\label{eq:symmetric-system}
Q
\begin{pmatrix}
x\\
\eta
\end{pmatrix}
=
\begin{pmatrix}
b\\
0
\end{pmatrix},
\qquad
\text{where }
Q
:=
\begin{pmatrix}
\widehat{A} & BZ\\
Z^TB^T & Z^TWZ
\end{pmatrix}.
\end{equation}
In particular, $Q$ is symmetric.
\end{proposition}

\begin{proof}
Since $\Lambda=\operatorname{range}(Z)$, we write
$\lambda=Z\eta$. Substituting this expression into the first equation
of Eq.~\eqref{eq:weak-mixed-global} gives
\begin{equation*}
\widehat{A}x+BZ\eta=b.
\end{equation*}

For the second equation of Eq.~\eqref{eq:weak-mixed-global}, let
$\mu=Z\xi$, with arbitrary $\xi\in\mathbb{R}^r$. Then
\begin{equation*}
\xi^TZ^TW(Z\eta-Mx)=0
\end{equation*}
for every $\xi\in\mathbb{R}^r$, and therefore
\begin{equation*}
Z^TWZ\eta-Z^TWMx=0.
\end{equation*}

Using Proposition~\ref{prop:BWM}, namely $B^T=-WM$, we obtain
\begin{equation*}
Z^TB^Tx+Z^TWZ\eta=0.
\end{equation*}

Combining the two equations yields
Eq.~\eqref{eq:symmetric-system}. The symmetry of $Q$ follows from the
symmetry of $\widehat{A}$ and $W$.
\end{proof}

The following identity relates the block-diagonal local operator
$\widehat{A}$ to the original matrix $A$.

\begin{lemma}
\label{lem:Ahat-identity}
The matrices $A$, $\widehat{A}$, $M$, and $W$ satisfy
\begin{equation}
\label{eq:Ahat-identity}
\widehat{A}
=
A+M^TWM.
\end{equation}
\end{lemma}

\begin{proof}
For any $x\in\mathbb{R}^{\Omega}$, using the decomposition of $\Omega$
into subdomains and the fact that $A_{ij}\neq0$ if and only if
$i\sim j$, the quadratic form associated with $A$ can be written as
\begin{equation}
\label{eq:A-quadratic-block}
x^TAx
=
\sum_{i=1}^m\sum_{j=1}^m
x_{\Omega_i}^TA_{ij}x_{\Omega_j}
=
\sum_{i=1}^m
x_{\Omega_i}^TA_{ii}x_{\Omega_i}
+
\sum_{i=1}^m
\sum_{j:\,j\sim i}
x_{\Omega_i}^TA_{ij}x_{\Omega_j}.
\end{equation}

On the other hand, using Eqs.~\eqref{eq:Ahat} and~\eqref{eq:Aiabs}, we obtain
\begin{equation}
\label{eq:Ahat-quadratic-block}
x^T\widehat{A}x
=
\sum_{i=1}^m
x_{\Omega_i}^TA_{\Omega_i}x_{\Omega_i}
=
\sum_{i=1}^m
x_{\Omega_i}^TA_{ii}x_{\Omega_i}
+
\sum_{i=1}^m
\sum_{j:\,j\sim i}
x_{\Omega_i}^TD_{ij}x_{\Omega_i}.
\end{equation}

Subtracting Eqs.~\eqref{eq:A-quadratic-block} and~\eqref{eq:Ahat-quadratic-block} gives
\begin{align}
\label{eq:Ahat-minus-A-block}
x^T(\widehat{A}-A)x
&=
\sum_{i=1}^m
\sum_{j:\,j\sim i}
\left(
x_{\Omega_i}^TD_{ij}x_{\Omega_i}
-
x_{\Omega_i}^TA_{ij}x_{\Omega_j}
\right)
\nonumber\\
&=
\sum_{\substack{i<j\\i\sim j}}
\Big(
x_{\Omega_i}^TD_{ij}x_{\Omega_i}
+
x_{\Omega_j}^TD_{ji}x_{\Omega_j}
-
x_{\Omega_i}^TA_{ij}x_{\Omega_j}
-
x_{\Omega_j}^TA_{ji}x_{\Omega_i}
\Big)
\nonumber\\
&=
\sum_{\substack{i<j\\i\sim j}}
\left(
x_{\Omega_i}^TD_{ij}x_{\Omega_i}
+
x_{\Omega_j}^TD_{ji}x_{\Omega_j}
-
2x_{\Omega_i}^TA_{ij}x_{\Omega_j}
\right),
\end{align}
where in the last equality we used the symmetry of $A$.

For each $i<j$ with $i\sim j$, using
Eq.~\eqref{eq:Dij-component} and the definition of
$\mathcal{E}_{ij}$, we have
\begin{equation}
\label{eq:Dij-quadratic}
x_{\Omega_i}^TD_{ij}x_{\Omega_i}
=
\sum_{(k,l)\in\mathcal{E}_{ij}}
|a_{kl}|x_k^2,
\qquad
x_{\Omega_j}^TD_{ji}x_{\Omega_j}
=
\sum_{(k,l)\in\mathcal{E}_{ij}}
|a_{kl}|x_l^2,
\end{equation}
where the symmetry of $A$ has been used in the second identity. Moreover,
\begin{equation}
\label{eq:Aij-component}
x_{\Omega_i}^TA_{ij}x_{\Omega_j}
=
\sum_{(k,l)\in\mathcal{E}_{ij}}
a_{kl}x_kx_l.
\end{equation}

Hence, substituting Eqs.~\eqref{eq:Dij-quadratic}
and~\eqref{eq:Aij-component} into Eq.~\eqref{eq:Ahat-minus-A-block},
and using the definition of $\mathcal{E}$ in Eq.~\eqref{eq:E}, we obtain
\begin{align*}
\label{eq:Ahat-minus-A-edge}
x^T(\widehat{A}-A)x
&=
\sum_{\substack{i<j\\i\sim j}}
\sum_{(k,l)\in\mathcal{E}_{ij}}
\left(
|a_{kl}|x_k^2
+
|a_{kl}|x_l^2
-
2a_{kl}x_kx_l
\right)
\nonumber\\
&=
\sum_{(k,l)\in\mathcal{E}}
|a_{kl}|
\left(
x_k-\operatorname{sign}(a_{kl})x_l
\right)^2
=
x^TM^TWMx.
\nonumber
\end{align*}

Since this identity holds for every $x\in\mathbb{R}^{\Omega}$ and both
$\widehat{A}-A$ and $M^TWM$ are symmetric, it follows that
Eq.~\eqref{eq:Ahat-identity} holds.
\end{proof}

We are now in a position to establish the well-posedness of the
AlgMortar formulation.

\begin{theorem}
\label{theorem:Q-psd}
Assume that $A$ is SPD. Then the matrix $Q$ defined in
Eq.~\eqref{eq:symmetric-system} is SPD. Consequently, the AlgMortar
formulation~\eqref{eq:weak-mixed-global} admits a unique solution.
\end{theorem}

\begin{proof}
Let $x\in\mathbb{R}^{\Omega}$ and $\eta\in\mathbb{R}^r$, and set
$\lambda=Z\eta$. From the definition of $Q$ in
Eq.~\eqref{eq:symmetric-system}, we obtain
\begin{equation*}
\begin{pmatrix}
x\\
\eta
\end{pmatrix}^T
Q
\begin{pmatrix}
x\\
\eta
\end{pmatrix}
=
x^T\widehat{A}x
+
2x^TB\lambda
+
\lambda^TW\lambda.
\end{equation*}

By Proposition~\ref{prop:BWM}, $B^T=-WM$. Since $W$ is symmetric,
it follows that $B=-M^TW$. Using Lemma~\ref{lem:Ahat-identity}, namely
$\widehat{A}=A+M^TWM$, we obtain
\begin{align}
\begin{pmatrix}
x\\
\eta
\end{pmatrix}^T
Q
\begin{pmatrix}
x\\
\eta
\end{pmatrix}
&=
x^TAx
+
(Mx)^TWMx
-
2(Mx)^TW\lambda
+
\lambda^TW\lambda
\nonumber\\
&=
x^TAx
+
(Mx-\lambda)^T
W
(Mx-\lambda).
\label{eq:Q-edge-form}
\end{align}

By the definition of $\mathcal{E}$ and Eq.~\eqref{eq:W}, $W$ is SPD.
Hence both terms on the right-hand side of
Eq.~\eqref{eq:Q-edge-form} are nonnegative. If $x\neq0$, then
$x^TAx>0$ because $A$ is SPD. If $x=0$ and $\eta\neq0$, then
$\lambda=Z\eta\neq0$ because the columns of $Z$ are linearly
independent, and Eq.~\eqref{eq:Q-edge-form} reduces to
$\lambda^TW\lambda>0$. Thus, the quadratic form is strictly positive
for every $(x,\eta)\neq(0,0)$, and therefore $Q$ is SPD.

Consequently, Eq.~\eqref{eq:symmetric-system}, and equivalently the
AlgMortar formulation~\eqref{eq:weak-mixed-global}, admits a unique
solution.
\end{proof}

\begin{corollary}
\label{cor-schur-eta}
Under the assumptions of Theorem~\ref{theorem:Q-psd}, the Schur
complement obtained by eliminating $x$ from
Eq.~\eqref{eq:symmetric-system} is SPD and is given by
\begin{equation}
\label{eq:S-eta}
S_{\eta}
=
Z^TS_{\lambda}Z,
\end{equation}
where
\begin{equation}
\label{k-schur}
S_{\lambda}
=
W-B^T\widehat{A}^{-1}B.
\end{equation}
\end{corollary}

\begin{proof}
Since $\widehat{A}$ and $Q$ are SPD, the Schur complement of
$\widehat{A}$ in $Q$ is also SPD
\cite[Proposition~14.1]{saad2003sparse}. Using the definition of $Q$ in
Eq.~\eqref{eq:symmetric-system}, this Schur complement is
\[
S_{\eta}
=
Z^TWZ
-
Z^TB^T\widehat{A}^{-1}BZ
=
Z^T
\left(
W-B^T\widehat{A}^{-1}B
\right)
Z,
\]
which proves the result.
\end{proof}

The matrix $S_{\lambda}$ in Eq.~\eqref{k-schur} is precisely the
operator appearing in the interface problem
\eqref{eq:interface-problem-alg-sym}. Therefore,
Corollary~\ref{cor-schur-eta} shows that the resulting interface system
is SPD on $\Lambda$.

\section{AlgMortar method as a preconditioner}
\label{sec:mortar-pc}

\subsection{Positive definiteness of the reduced system after interface variable elimination}
\label{sec:mortar-pc-positive}

We now consider AlgMortar as a PCG preconditioner. Eliminating the
interface variable yields a reduced SPD operator, as established in the
following corollary.

\begin{corollary}
\label{cor-schur-p}
Under the same assumptions as in Theorem~\ref{theorem:Q-psd}, the
Schur complement of $Q$ associated with the elimination of the
interface variable $\eta$ is SPD. Consequently, the reduced system can
be written as
\begin{equation*}
\mathcal{M}x=b,
\end{equation*}
where
\begin{equation}
\label{eq-m-reduced-p}
\mathcal{M}
=
\widehat{A}
-
BZ
\left(
Z^TWZ
\right)^{-1}
Z^TB^T.
\end{equation}
In particular, $\mathcal{M}$ is SPD.
\end{corollary}

\begin{proof}
Since $Q$ is SPD by Theorem~\ref{theorem:Q-psd}, and $Z^TWZ$ is SPD,
the Schur complement associated with the elimination of $\eta$ is also
SPD; see, e.g., \cite[Proposition~14.1]{saad2003sparse}. Eliminating
$\eta$ from Eq.~\eqref{eq:symmetric-system} gives
Eq.~\eqref{eq-m-reduced-p}.
\end{proof}

\subsection{Spectral analysis of the preconditioned system}
\label{sec:mortar-pc-spectrum}

We now analyze the convergence of PCG when $\mathcal{M}$ is used as a
preconditioner for $A$. Since both $A$ and $\mathcal{M}$ are SPD, the
convergence of PCG is governed by the condition number
$\kappa_2(\mathcal{M}^{-1/2}A\mathcal{M}^{-1/2})$, where
$\kappa_2(\cdot)$ denotes the condition number induced by the Euclidean
norm; see, e.g., \cite[Section~14.5]{strikwerda2004finite}. The closer
this condition number is to one, the more effective the preconditioner
and, consequently, the faster the expected convergence of PCG.

For the analysis below, we also introduce the energy norm on
$\mathbb{R}^{\Omega}$, defined by $\|x\|_A^2:=x^TAx$, and the
$W$-norm on $\mathbb{R}^{\mathcal{E}}$, defined by
$\|\mu\|_W^2:=\mu^TW\mu$. We define
\begin{equation}
\label{eq:projection-pi-lambda}
\Pi_{\Lambda}
:=
Z(Z^TWZ)^{-1}Z^TW,
\end{equation}
which is the $W$-orthogonal projection from
$\mathbb{R}^{\mathcal{E}}$ onto
$\Lambda=\operatorname{range}(Z)$.

Using Proposition~\ref{prop:BWM} and
Lemma~\ref{lem:Ahat-identity}, the reduced operator
in Eq.~\eqref{eq-m-reduced-p} can be written as
\begin{equation}
\label{eq:m-reduced-projection}
\mathcal{M}
=
A+M^TW(I-\Pi_{\Lambda})M.
\end{equation}

\begin{lemma}
\label{lemma:projection-lambda}
For every $\mu\in\mathbb{R}^{\mathcal{E}}$,
\begin{equation}
\label{eq:projection-complement-norm-identity}
\|(I-\Pi_{\Lambda})\mu\|_W^2
=
\mu^TW(I-\Pi_{\Lambda})\mu.
\end{equation}
\end{lemma}

\begin{proof}
From Eq.~\eqref{eq:projection-pi-lambda}, we have
$\Pi_{\Lambda}^2=\Pi_{\Lambda}$ and
\[
\Pi_{\Lambda}^TW=W\Pi_{\Lambda}.
\]
Consequently,
$(I-\Pi_{\Lambda})^TW=W(I-\Pi_{\Lambda})$. Therefore,
\begin{align*}
\|(I-\Pi_{\Lambda})\mu\|_W^2
&=
\mu^T(I-\Pi_{\Lambda})^TW(I-\Pi_{\Lambda})\mu
\\
&=
\mu^TW(I-\Pi_{\Lambda})^2\mu = \mu^TW(I-\Pi_{\Lambda})\mu,
\end{align*}
which proves the result.
\end{proof}

To analyze the eigenvalues of $\mathcal{M}^{-1/2}A\mathcal{M}^{-1/2}$, following \cite{zhou2026algebraicmultiscalepreconditionerlarge}, we first observe that this matrix is SPD and similar to $\mathcal{M}^{-1}A$. Indeed,
\begin{equation*}
\mathcal{M}^{-1}A
=
\mathcal{M}^{-1/2}
\left(
\mathcal{M}^{-1/2}A\mathcal{M}^{-1/2}
\right)
\mathcal{M}^{1/2}.
\end{equation*}

Therefore, the two matrices have the same eigenvalues. Hence, it is
sufficient to analyze the spectrum of $\mathcal{M}^{-1}A$.

\begin{proposition}
\label{prop:bound-eigenvalues}
Let $\theta$ be an eigenvalue of $\mathcal{M}^{-1}A$. Then
\begin{equation}
\label{prop:bound-eigenvalues-eq}
\frac{1}{1+K_{\Lambda}}
\leq
\theta
\leq
1,
\end{equation}
where
\begin{equation}
\label{eq:k-lambda}
K_{\Lambda}
=
\max_{x\neq0}
\frac{
\|(I-\Pi_{\Lambda})Mx\|_W^2
}{
\|x\|_A^2
}.
\end{equation}

Consequently,
\begin{equation}
\label{eq:condition-bound}
\kappa_2
\left(
\mathcal{M}^{-1/2}A\mathcal{M}^{-1/2}
\right)
\leq
1+K_{\Lambda}.
\end{equation}
\end{proposition}
\begin{proof}

Using Eq.~\eqref{eq:m-reduced-projection} and
Lemma~\ref{lemma:projection-lambda}, for every
$x\in\mathbb{R}^{\Omega}$ we obtain
\begin{align*}
x^T\mathcal{M}x
&=
x^TAx
+
(Mx)^TW(I-\Pi_{\Lambda})Mx
\\
&=
x^TAx
+
\|(I-\Pi_{\Lambda})Mx\|_W^2.
\end{align*}

Therefore, by the definition of $K_{\Lambda}$,
\begin{equation}
\label{prop:bound-eigenvalues-ineq}
x^TAx
\leq
x^T\mathcal{M}x
\leq
(1+K_{\Lambda})x^TAx.
\end{equation}

Let $(q,\theta)$ be an eigenpair of $\mathcal{M}^{-1}A$. Then
$Aq=\theta\mathcal{M}q$, and hence
\begin{equation*}
\theta
=
\frac{q^TAq}{q^T\mathcal{M}q}.
\end{equation*}

Applying Eq.~\eqref{prop:bound-eigenvalues-ineq} with $x=q$ yields
Eq.~\eqref{prop:bound-eigenvalues-eq}.

Since $\mathcal{M}^{-1}A$ and
$\mathcal{M}^{-1/2}A\mathcal{M}^{-1/2}$ have the same eigenvalues,
and the latter is SPD, Eq.~\eqref{prop:bound-eigenvalues-eq} also gives
Eq.~\eqref{eq:condition-bound}.
\end{proof}

\begin{remark}
\label{rem:K-lambda}
The quantity $K_{\Lambda}$ measures the component of the signed
interface quantity $Mx$ that is not represented by the interface space
$\Lambda$, measured in the $W$-norm and normalized by the $A$-norm of
$x$. Hence, smaller values of $K_{\Lambda}$ correspond to better
approximation by the interface space, tighter eigenvalue clustering
around one, and a smaller spectral condition number.
\end{remark}

\subsection{Setup and application of the AlgMortar preconditioner}
\label{sec:mortar-pc-setup-apply}

The inverse of the matrix in Eq.~\eqref{eq-m-reduced-p} defines the
proposed preconditioner. In practice, however, we avoid assembling
$\mathcal{M}$ explicitly and apply $\mathcal{M}^{-1}$ through the MBFs.

We split the implementation into two phases: the setup, presented in
Algorithm~\ref{alg:amm-preconditioner-setup}, and the application,
presented in Algorithm~\ref{alg:amm-preconditioner-application}. The
setup is performed once before the PCG iteration starts, whereas the
application is performed at each PCG iteration using the current
residual as the RHS. In practice, the system is permuted once during setup, PCG is applied
in this ordering, and the converged solution is mapped back to the
original ordering.

\begin{algorithm}[!ht]
\caption{Construction of the AlgMortar preconditioner}
\label{alg:amm-preconditioner-setup}
\begin{algorithmic}[1]

\STATE{\textbf{function} AlgMortarSetup$(A,m,\{\psi_k\}_{k=1}^{N_\Lambda})$}

\STATE Partition the graph of $A$ into $m$ subdomains
\STATE Compute the corresponding permutation matrix $P$
\STATE Construct the permuted matrix $A\gets PAP^T$

\FORALL{$i=1,\ldots,m$}
    \STATE Extract $A_{ii}$ and construct $A_{\Omega_i}$ according to
    Eq.~\eqref{eq:Aiabs}
    \STATE Compute the Cholesky factorization
    $A_{\Omega_i}=L_iL_i^T$

    \FORALL{$k\in\mathcal{I}_i$}
        \STATE Compute the MBF $\widehat{x}_{\Omega_i}^{(k)}$ by solving
        $L_iL_i^T\widehat{x}_{\Omega_i}^{(k)}=-B_i\psi_k$
    \ENDFOR
\ENDFOR

\STATE Assemble the interface system matrix $S$ by testing
Eq.~\eqref{eq:interface-problem-alg-sym} with the basis functions
$\{\psi_k\}_{k=1}^{N_\Lambda}$

\STATE{\textbf{end function}}

\end{algorithmic}
\end{algorithm}

\begin{algorithm}[!ht]
\caption{Application of the AlgMortar preconditioner}
\label{alg:amm-preconditioner-application}
\begin{algorithmic}[1]

\STATE{\textbf{function} AlgMortarApply$(r,
\{L_i\}_{i=1}^m,
\{\widehat{x}_{\Omega_i}^{(k)}\}_{i=1}^m,S)$}

\FORALL{$i=1,\ldots,m$}
    \STATE Extract the local residual $r_{\Omega_i}$
    \STATE Compute the local particular solution
    $\overline{z}_{\Omega_i}$ by solving
    $L_iL_i^T\overline{z}_{\Omega_i}=r_{\Omega_i}$
\ENDFOR

\STATE Assemble the interface RHS $g$ from
$\{\overline{z}_{\Omega_i}\}_{i=1}^m$
according to Eq.~\eqref{eq:interface-problem-alg-sym}

\STATE Solve the interface system $S\lambda=g$

\FORALL{$i=1,\ldots,m$}
    \STATE Reconstruct the local correction
    $z_{\Omega_i}
    =
    \overline{z}_{\Omega_i}
    +
    \displaystyle\sum_{k\in\mathcal{I}_i}
    \lambda_k\widehat{x}_{\Omega_i}^{(k)}$
\ENDFOR

\STATE Assemble $z$ from $\{z_{\Omega_i}\}_{i=1}^m$
\STATE \RETURN $z$

\STATE{\textbf{end function}}

\end{algorithmic}
\end{algorithm}

\subsection{Smoothing}
\label{sec:smoothing}

Let $x$ denote the AlgMortar approximation for a given RHS $b$, and
let $r=b-Ax$
be the corresponding residual. Using Eq.~\eqref{eq:Bij-identity}, the
restriction of $Ax$ to $\Omega_i$ can be written as
\begin{equation*}
(Ax)_{\Omega_i}
=
A_{\Omega_i}x_{\Omega_i}+B_iMx.
\end{equation*}
On the other hand, the first equation in
Eq.~\eqref{eq:weak-mixed-global} gives
\begin{equation*}
A_{\Omega_i}x_{\Omega_i}+B_i\lambda=b_{\Omega_i}.
\end{equation*}
Therefore,
\begin{equation}
\label{eq:interface-residual}
r_{\Omega_i}
=
B_i(\lambda-Mx)
=
\sum_{j:\,j\sim i}B_{ij}(\lambda-Mx).
\end{equation}

It follows from the definition of $B_{ij}$ in Eq.~\eqref{eq:Bij} that
$(r_{\Omega_i})_k=0$ for every $k\in\Omega_i$ that is not incident to
an interface edge. Hence, the AlgMortar residual is supported only on
unknowns adjacent to the graph-partition interfaces. These residual
components can be further reduced by applying a smoothing procedure.

Smoothing strategies have also been combined with multiscale
preconditioners to reduce high-frequency error components
\cite{arbogast2015twolevel,yang2019coarsespace,bosma2021msrsb,
fu2024preconditioner,chabi2024monotone,carvalho2025preconditioner,
vasilyeva2025heatflow,changqing2025highconstrast}.
We use block Jacobi as both a pre- and post-smoother, since among the
smoothing strategies considered in the numerical experiments of
Section~\ref{sec:num-exp}, it yielded the lowest computational times.

\subsection{Interface space}
\label{sec:mortar-pc-interface}

To apply AlgMortar, an interface space $\Lambda$ must be specified. In this work, we construct $\Lambda$ by splitting each interface according to the sign of the corresponding matrix couplings:
\begin{equation}
\label{eq:interface-sign-split}
\mathcal{E}_{ij}^{-}
:=
\left\{
(k,l)\in\mathcal{E}_{ij}:a_{kl}<0
\right\},
\qquad
\mathcal{E}_{ij}^{+}
:=
\left\{
(k,l)\in\mathcal{E}_{ij}:a_{kl}>0
\right\}.
\end{equation}
For each nonempty set $\mathcal{E}_{ij}^{\sigma}$,
$\sigma\in\{-,+\}$, we define a basis function
$\psi_{ij}^{\sigma}$ that is equal to one on
$\mathcal{E}_{ij}^{\sigma}$ and zero elsewhere. The interface space
$\Lambda$ is then defined as the span of all such basis functions.

If all couplings across an interface have the same sign, the usual piecewise constant interface space is recovered. When both positive and negative couplings are present, the sign-split construction uses separate basis functions for each sign. In our numerical experiments, we observed that this construction reduces the number of PCG iterations and the overall solution time, despite the larger interface systems.

\section{Implementation details}\label{sec:mortar-pc-algorithm}
The AlgMortar method proposed in this work was implemented using the Portable, Extensible Toolkit for Scientific Computation (PETSc)  \cite{balay2025petsc}, through its Python interface, \texttt{petsc4py} \footnote{\url{https://petsc.org/release/petsc4py/}}. 
PETSc supports parallel execution through the Message Passing Interface (MPI) and provides scalable solvers for high-performance scientific computing environments.

The proposed preconditioner was implemented using the PETSc PCShell interface. Block-Jacobi smoothing is incorporated through a multiplicative composition with PETSc's built-in block-Jacobi preconditioner, using its default configuration and specifying only the number of blocks. This number is typically smaller than the number of AlgMortar subdomains, producing larger blocks and more effective local smoothing. This choice performed well in the experiments of Section~\ref{sec:num-exp}.

For graph partitioning, we first use ParMETIS \cite{parmetis2013} through PETSc to divide the matrix graph into macro-partitions, one per MPI process. METIS is then applied locally within each macro-partition to construct the AlgMortar subdomains. Similar strategies appear in \cite{zhou2026algebraicmultiscalepreconditionerlarge}. In our experiments, this two-stage approach provided better load balancing and scalability than a single global ParMETIS partition.

To compute the MBFs, we follow \cite{jaramillo2022billion,carvalho2025preconditioner} and solve the local problems with the MUMPS direct solver \cite{mumps2001} through PETSc. During setup, MUMPS computes a Cholesky factorization of each local matrix using Approximate Minimum Degree (AMD) ordering, and these factorizations are reused during the application of the preconditioner. For the interface system, we consider two approaches: a direct MUMPS solve with Cholesky factorization and PORD ordering, following \cite{jaramillo2022billion,carvalho2025preconditioner}, and an iterative PCG solve with a small relative residual tolerance and block Jacobi preconditioning, which provided the best performance in our experiments.

In Section~\ref{sec:num-exp}, we compare the proposed method with the algebraic multigrid (AMG) method \cite{stuben2001amg}, using the BoomerAMG implementation \cite{henson2002amg} from the HYPRE library, available through PETSc. We adopt the BoomerAMG parameters reported in Table~1 of \cite{jaramillo2025porenetwork}, which, in our experiments, yielded substantially lower computational times than the default BoomerAMG configuration available through PETSc.

\section{Numerical experiments}
\label{sec:num-exp}
In this section, we evaluate AlgMortar as a PCG preconditioner. We consider pressure systems arising from fine-grid, cell-centered finite-volume discretizations of the single-phase Darcy flow problem (Eq.~\eqref{eq:elliptic-problem}) using TPFA in Subsections~\ref{sec:scal-cross}-\ref{sec:spe11}, as well as non-TPFA matrices in Subsection~\ref{sec:suitesparse}.

PCG is initialized with the zero vector. Unless otherwise stated, the interface problem is solved using PCG with block-Jacobi preconditioning and a relative tolerance of $10^{-2}$. AlgMortar uses 400 subdomains per MPI process, and block-Jacobi smoothing is applied with 80 blocks per MPI process.
The tests presented in this section were carried out on the Euler supercomputer\footnote{\url{https://euler.cemeai.icmc.usp.br}}, hosted by the Center for Mathematical Sciences Applied to Industry (CeMEAI) at the University of São Paulo, Brazil.

\subsection{Scalability of AlgMortar for a high-contrast cross-shaped permeability field}
\label{sec:scal-cross}

We assess the scalability of AlgMortar using the isotropic high-contrast    
permeability field from \cite{changqing2025highconstrast}. The domain
$\Omega=(0,1)^3$ is discretized by a uniform Cartesian grid. The permeability
tensor is $K=\kappa I$, where $\kappa=10^6$ along a cross-shaped pattern defined
on a base $8^3$ grid and $\kappa=1$ elsewhere, as shown in
Figure~\ref{fig:cross_perm_8}. Periodic replication of this pattern forms
connected high-permeability channels, illustrated in
Figures~\ref{fig:cross_perm_16} and~\ref{fig:cross_perm_24}. We impose $p=0$ on
$y=0$, $p=1$ on $y=1$, homogeneous Neumann conditions elsewhere, and a zero
source term.
In our PETSc implementation, the matrix is initially distributed across MPI processes according to the domain decomposition provided by PETSc \texttt{DMDA}.

\begin{figure*}[!ht]
    \centering

    \begin{subfigure}[b]{0.30\textwidth}
        \centering
        \includegraphics[width=\textwidth]{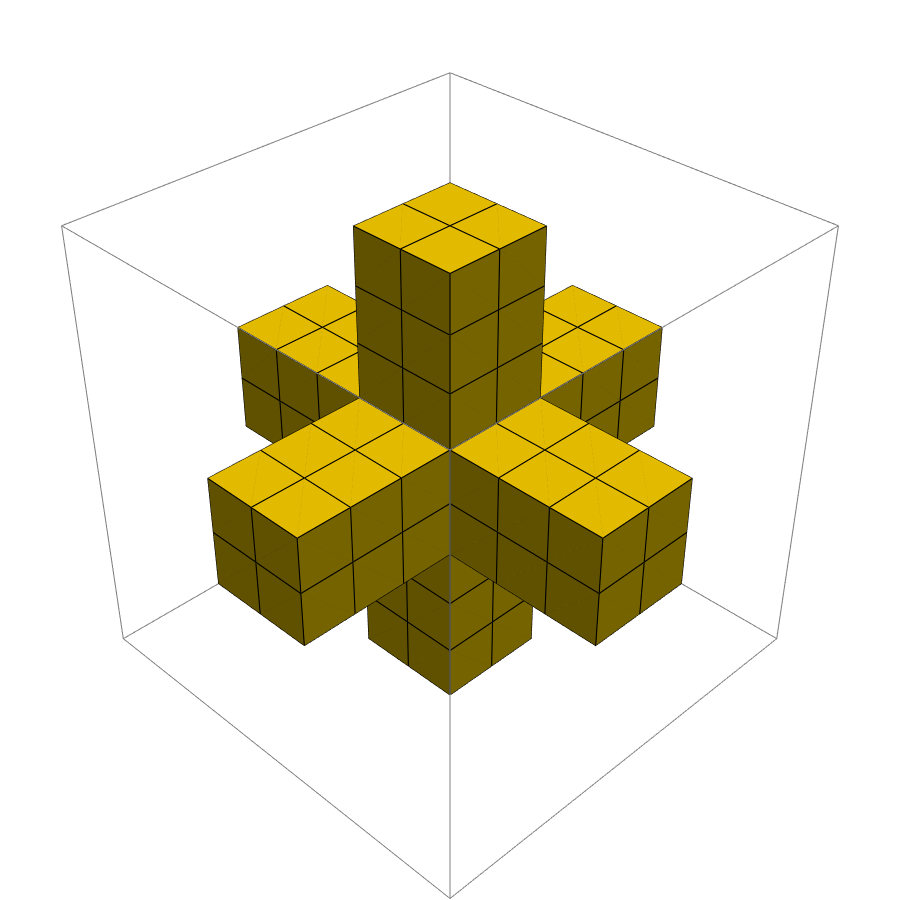}
        \caption{$8^3$ grid.}
        \label{fig:cross_perm_8}
    \end{subfigure}
    \hfill
    \begin{subfigure}[b]{0.30\textwidth}
        \centering
        \includegraphics[width=\textwidth]{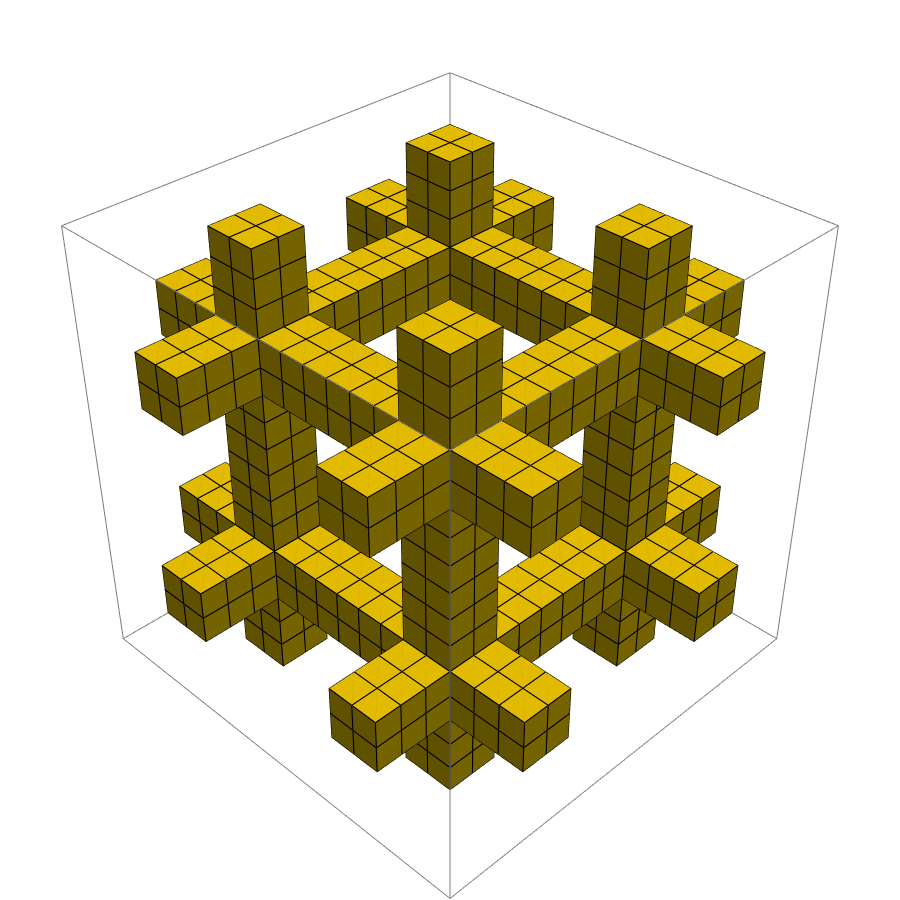}
        \caption{$16^3$ grid.}
        \label{fig:cross_perm_16}
    \end{subfigure}
    \hfill
    \begin{subfigure}[b]{0.30\textwidth}
        \centering
        \includegraphics[width=\textwidth]{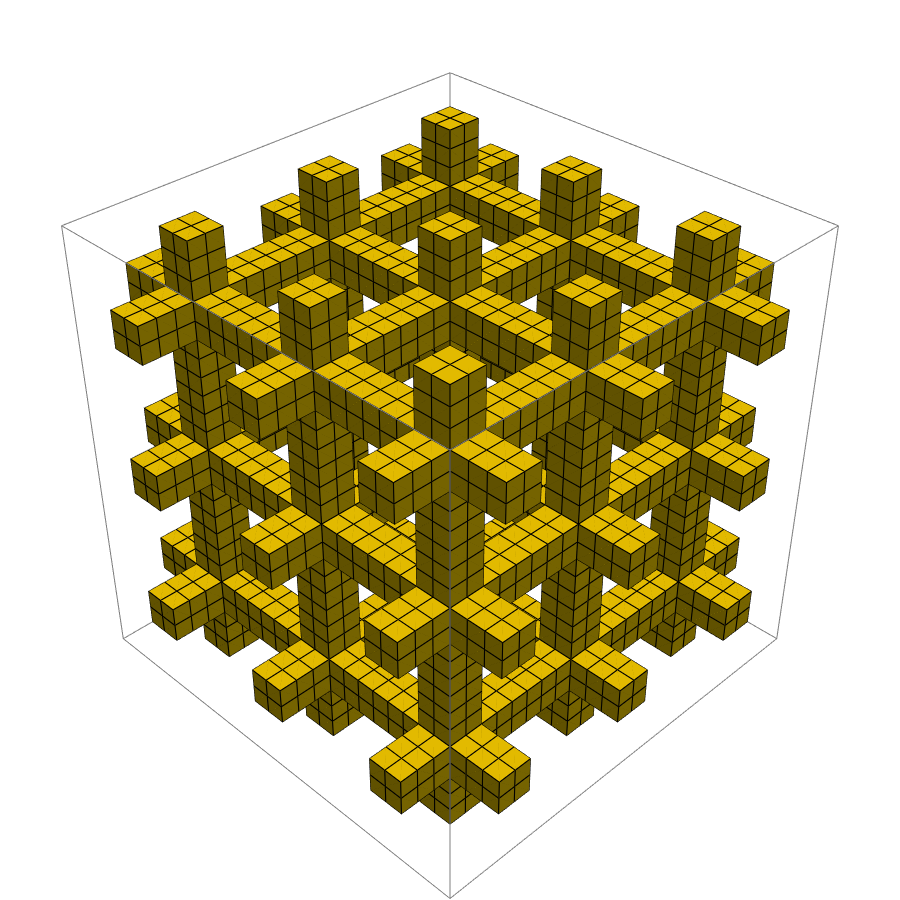}
        \caption{$24^3$ grid.}
        \label{fig:cross_perm_24}
    \end{subfigure}

    \caption{Illustration of the cross-shaped permeability field proposed by \cite{changqing2025highconstrast}:
    a base cross-shaped pattern on an $8^3$ grid and its periodic
    repetitions on $16^3$ and $24^3$ grids.}
    \label{fig:cross_perm}
\end{figure*}

\subsubsection{Weak scaling analysis}
\label{sec:weak-cross}
We consider Cartesian grids of size $(80r)\times(80r)\times(80r)$ distributed over $r^3$ MPI processes, with $r=3,4,5,6$. Thus, the number of degrees of freedom per process remains constant, defining a weak-scaling study. We set the relative residual tolerance ($\mathrm{rtol}$), measured in the Euclidean norm, to $10^{-11}$.

We compare AlgMortar using two strategies for the
interface system: a direct solve with MUMPS and an iterative solve with PCG
preconditioned by block Jacobi, using a relative residual tolerance of
$10^{-2}$ for the interface solve. The corresponding weak-scaling results are
reported in Table~\ref{tab:weak_scaling_total_time_iterations_cross}. As a
reference, we also include results obtained by solving the fine-grid problem
with PCG preconditioned by AMG.

\begin{table}[htbp]
\centering
\footnotesize
\caption{Weak-scaling results for the high-contrast cross-shaped permeability field.}
\label{tab:weak_scaling_total_time_iterations_cross}

\resizebox{\linewidth}{!}{
\begin{tabular}{|c|c|c|c|c|c|}
\hline
\multirow{3}{*}{\# MPI proc.}
& \multirow{3}{*}{Global DOFs}
& \multirow{3}{*}{Interface size}
& \multicolumn{3}{c|}{Solution time in seconds (PCG iterations)} \\
\cline{4-6}
&
&
& \multicolumn{2}{c|}{AlgMortar}
& \multirow{2}{*}{AMG} \\
\cline{4-5}
&
&
& \shortstack{PCG-block Jacobi\\interface solver}
& \shortstack{MUMPS\\interface solver}
& \\
\hline
$27$
& $13{,}824{,}000$
& $75{,}149$
& $10.56\ (38)$
& $11.72\ (38)$
& $8.62\ (21)$ \\
$64$
& $32{,}768{,}000$
& $180{,}710$
& $12.51\ (37)$
& $15.29\ (36)$
& $15.13\ (22)$ \\
$125$
& $64{,}000{,}000$
& $358{,}820$
& $18.59\ (37)$
& $25.94\ (36)$
& $25.47\ (23)$ \\
$216$
& $110{,}592{,}000$
& $624{,}715$
& $21.42\ (37)$
& $39.65\ (36)$
& $40.02\ (24)$ \\
\hline
\end{tabular}}
\end{table}

Table~\ref{tab:weak_scaling_total_time_iterations_cross} shows that the interface solver has little effect on the outer PCG iteration count, indicating that the iterative approach preserves the quality of the AlgMortar preconditioner. However, PCG with block Jacobi consistently reduces the total solution time compared with the direct solver. 
Table~\ref{tab:weak_scaling_total_time_iterations_cross} also compares AlgMortar with AMG. Although AMG is faster on $27$ MPI processes, AlgMortar with the iterative interface solver is faster for all larger process counts. Its total time increases from $10.56$ s to $21.42$ s as the number of processes grows from $27$ to $216$, whereas the AMG time increases from $8.62$ s to $40.02$ s. Thus, AlgMortar with the iterative interface solver exhibits the best weak-scaling behavior among the methods considered.
\begin{figure*}[!ht]
    \centering

    \begin{subfigure}[b]{0.45\textwidth}
        \centering
        \includegraphics[width=\textwidth]
        {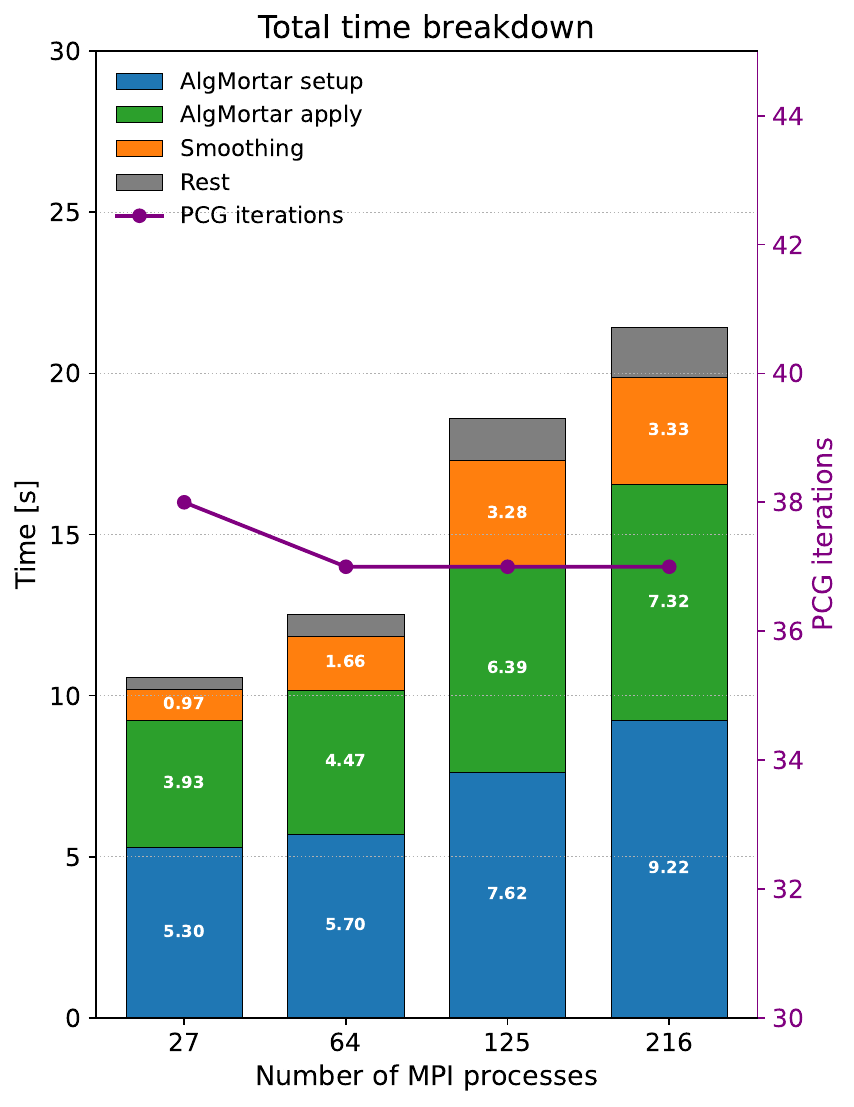}
        \caption{Total solution time.}
        \label{fig:cross_weak_scaling1}
    \end{subfigure}
    \hfill
    \begin{subfigure}[b]{0.45\textwidth}
        \centering
        \includegraphics[width=\textwidth]
        {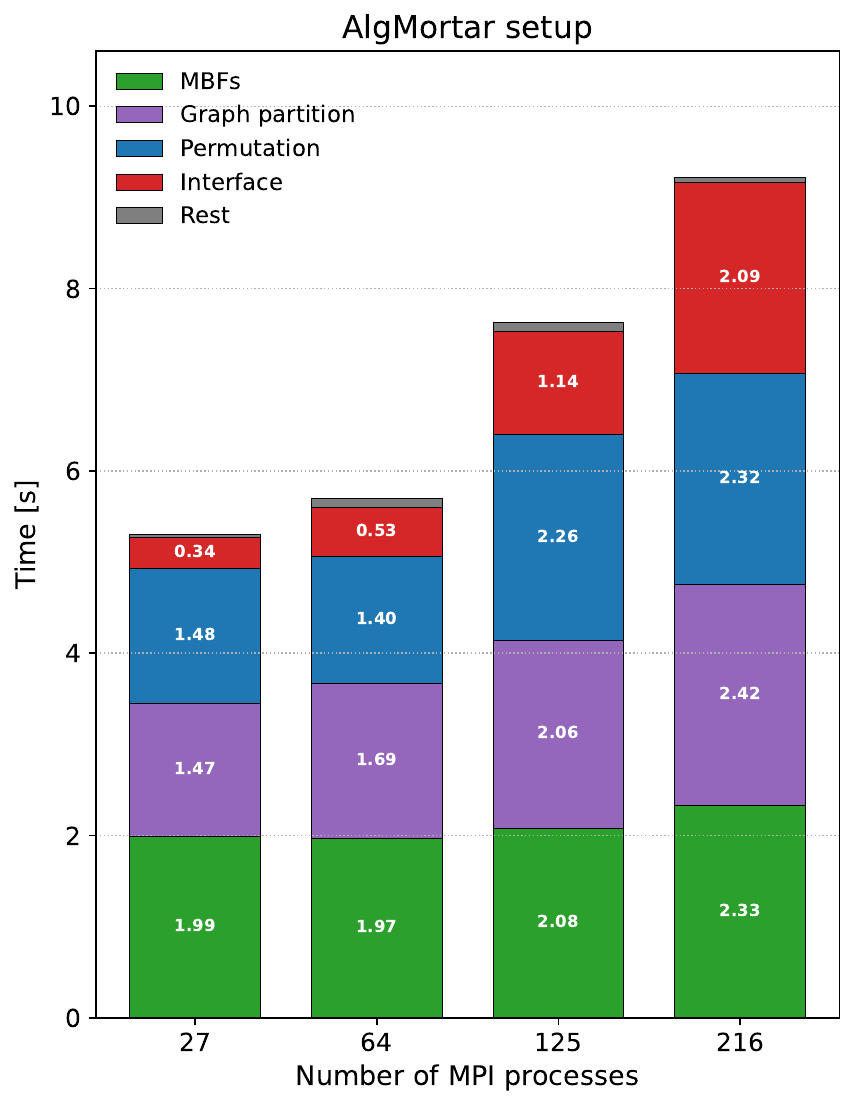}
        \caption{Setup time.}
        \label{fig:cross_weak_scaling2}
    \end{subfigure}

\caption{Weak-scaling time breakdown for AlgMortar on the high-contrast cross-shaped permeability field: (a) total time and (b) setup time.}
    \label{fig:cross_weak_time_breakdown}
\end{figure*}

Figure~\ref{fig:cross_weak_scaling1} presents the total-time breakdown, showing that the increase in the AlgMortar time with the iterative interface solver is mainly due to the setup phase.  The application time also increases, primarily because of interface-related operations. Figure~\ref{fig:cross_weak_scaling2} provides a detailed breakdown of this phase, in which the interface-system component is the dominant cost, followed by graph partitioning and permutation. In an existing parallel solver, the macro-partition may already be available, reducing this overhead to local micro-partitioning. Moreover, for time-dependent problems with fixed grid connectivity, the partitioning can be reused across time steps, further reducing this cost.

The weak-scaling results show that AlgMortar with the iterative interface solver exhibits good scalability and outperforms AMG in total solution time, with an increasing advantage for larger problems. Therefore, we adopt the iterative interface solver in all subsequent experiments.

\subsubsection{Strong scaling analysis}
\label{sec:strong-cross}

Next, we assess the strong scalability of AlgMortar for the cross-shaped permeability field. The fixed problem has $480\times480\times480$ cells, or $110{,}592{,}000$ degrees of freedom, while the number of MPI processes increases from $3^3$ to $6^3$. We set $\mathrm{rtol}=10^{-11}$.

\begin{table}[htbp]
\footnotesize
\caption{Strong-scaling results for AlgMortar on the cross-shaped permeability field.}
\label{tab:strong_scaling_total_time_cross}
\begin{center}
\begin{tabular}{|c|c|c|c|c|c|}
\hline
{\# MPI proc.}
& {Time (s)}
& {PCG it.}
& \shortstack{{Observed}\\{speedup}}
& \shortstack{{Ideal}\\{speedup}}
& \shortstack{{Parallel}\\{efficiency}} \\
\hline
$27$  & $120.35$ & $48$ & --     & --     & --     \\
$64$  & $50.40$  & $43$ & $2.39$ & $2.37$ & $1.01$ \\
$125$ & $33.49$  & $40$ & $3.59$ & $4.63$ & $0.78$ \\
$216$ & $21.42$  & $37$ & $5.62$ & $8.00$ & $0.70$ \\
\hline
\end{tabular}
\end{center}
\end{table}

\begin{figure*}[!ht]
\centering
\begin{subfigure}[b]{0.45\textwidth}
    \centering
    \includegraphics[width=\textwidth]
    {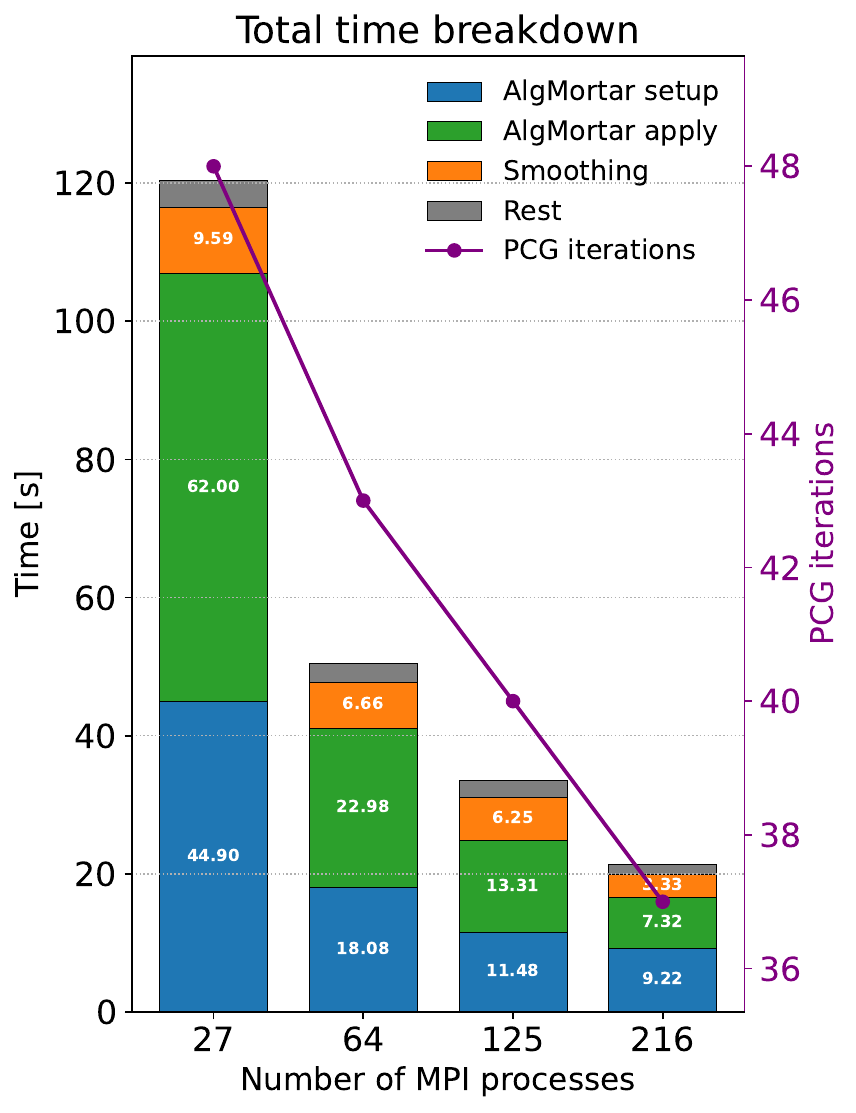}
    \caption{Total time.}
    \label{fig:cross_strong_scaling1}
\end{subfigure}
\hfill
\begin{subfigure}[b]{0.45\textwidth}
    \centering
    \includegraphics[width=\textwidth]
    {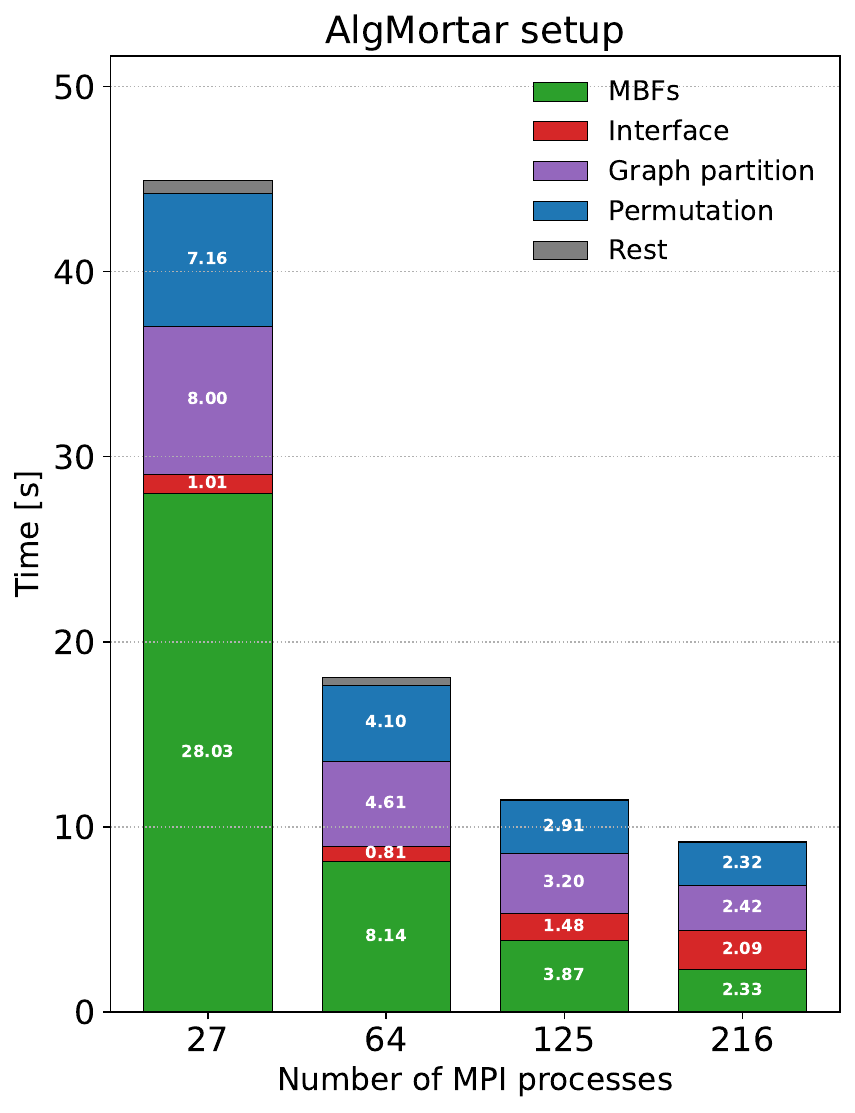}
    \caption{Setup time.}
    \label{fig:cross_strong_scaling2}
\end{subfigure}

\caption{Strong-scaling time breakdown for AlgMortar on the high-contrast cross-shaped permeability field: (a) total time and (b) setup time.}
\label{fig:cross_strong_scaling}
\end{figure*}

Table~\ref{tab:strong_scaling_total_time_cross} reports the solution time, PCG iterations, speedup, and parallel efficiency. The total time decreases from $120.35$ s on $27$ MPI processes to $21.42$ s on $216$, yielding a speedup of $5.62$ and an efficiency of $0.70$. The efficiency is nearly ideal at $64$ processes and remains at $0.70$ for the largest configuration, indicating good strong scalability of AlgMortar.

Figure~\ref{fig:cross_strong_scaling1} shows that both setup and application times decrease as the number of MPI processes increases, confirming the strong-scaling behavior reported in Table~\ref{tab:strong_scaling_total_time_cross}. The application phase dominates at lower process counts, while the setup phase becomes relatively more important at higher process counts. Figure~\ref{fig:cross_strong_scaling2} shows that the MBF computation scales well, whereas the interface component becomes the main setup bottleneck.

\subsection{Weak-scaling analysis of AlgMortar on the SPE10 benchmark}
\label{sec:scalability-spe10}

To assess scalability for a highly heterogeneous and representative porous-media flow problem, we consider the SPE10 benchmark \cite{christie2001spe10}. Following \cite{carvalho2025preconditioner}, we take $\Omega=(0,1)\times(0,2)\times(0,1)$, discretized by a $60\times120\times60$ uniform Cartesian grid. The isotropic permeability tensor is $K=\kappa I$, where $\kappa$ is the $x$-direction permeability from the first $60\times120\times60$ SPE10 cells. The permeability field is shown in Fig.~\ref{fig:perm_spe10_60x120x60}. We impose $p=0$ on $y=0$, $p=1$ on $y=2$, homogeneous Neumann conditions elsewhere, and a zero source term. Following \cite{jaramillo2022billion,carvalho2025preconditioner}, the base-grid permeability is transferred to the fine grids by an $L^2$ projection, preserving a contrast of $2.80\times10^7$. For the weak-scaling study, we use grids of size $(60r)\times(120r)\times(60r)$ distributed over $r^3$ MPI processes, with $r=3,4,5,6$.

\begin{figure*}[htbp]
    \centering

    \begin{subfigure}[b]{0.43\textwidth}
        \centering
        \includegraphics[width=\textwidth]
        {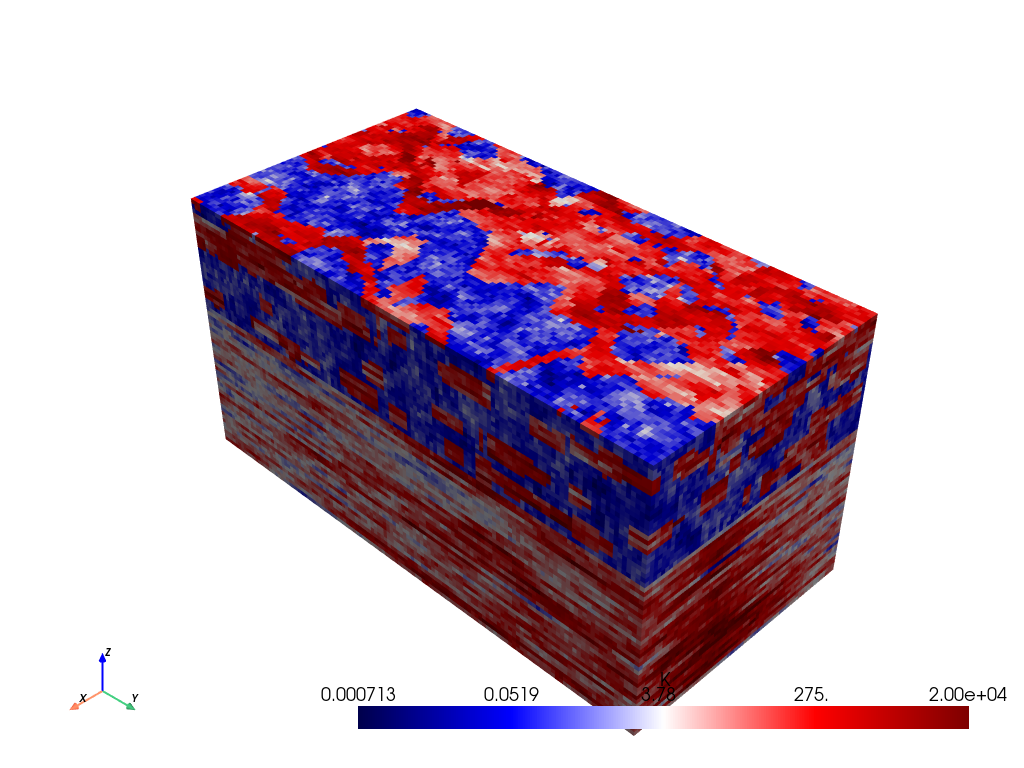}
        \caption{ $\kappa$ in $\log_{10}$ scale.}
        \label{fig:perm_spe10_60x120x60}
    \end{subfigure}
    \begin{subfigure}[b]{0.55\textwidth}
        \centering
        \includegraphics[width=\textwidth]
        {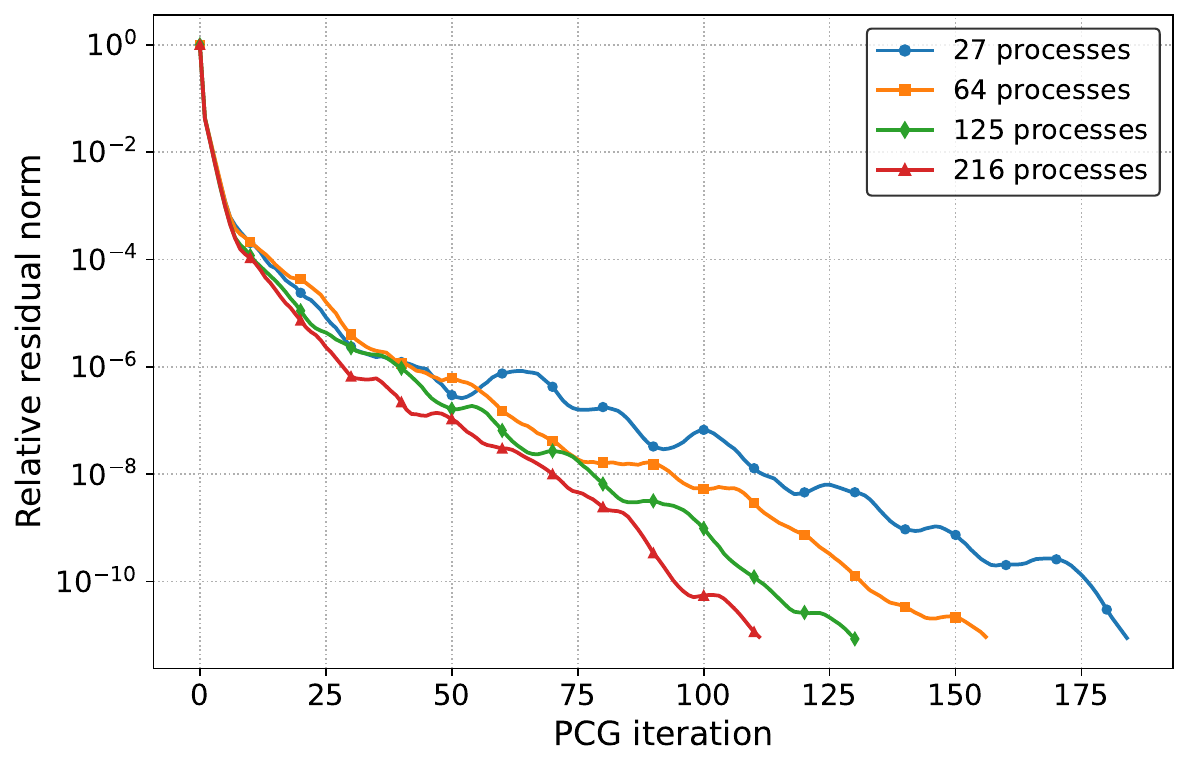}
        \caption{PCG residual histories.}
        \label{fig:weak-spe10-res}
    \end{subfigure}

    \caption{SPE10 test case: (a) horizontal permeability field
    $\log_{10}(\kappa)$ for the first $60\times120\times60$ cells and
    (b) PCG residual histories for the weak-scaling experiments.}
    \label{fig:spe10_permeability_residual_history}
\end{figure*}

Figure~\ref{fig:weak-spe10-res} shows the residual norm histories for the weak-scaling problems. For all MPI configurations, the residual decreases rapidly to approximately $10^{-6}$ within relatively few PCG iterations. Convergence then slows and becomes more configuration-dependent, requiring many additional iterations to reach $10^{-11}$. We therefore report results for both tolerances, since $10^{-6}$ may be sufficient in some practical applications.

\begin{table}[htbp]
\centering
\footnotesize
\caption{Weak-scaling results for SPE10. Each entry reports the total solution time in seconds, with the number of PCG iterations in parentheses.}
\label{tab:weak_scaling_total_time_iterations_spe10}
\begin{tabular}{|c|c|c|c|c|c|}
\hline
\multirow{3}{*}{\# MPI proc.}
& \multirow{3}{*}{Global DOFs}
& \multicolumn{4}{c|}{Solution time in seconds (PCG iterations)} \\
\cline{3-6}
&
& \multicolumn{2}{c|}{$\mathrm{rtol}=10^{-11}$}
& \multicolumn{2}{c|}{$\mathrm{rtol}=10^{-6}$} \\
\cline{3-6}
&
& AlgMortar
& AMG
& AlgMortar
& AMG \\
\hline
$27$
& $11{,}664{,}000$
& $26.21\ (184)$
& $10.51\ (29)$
& $9.96\ (44)$
& $8.73\ (14)$ \\
$64$
& $27{,}648{,}000$
& $29.66\ (156)$
& $16.01\ (30)$
& $12.01\ (42)$
& $12.72\ (15)$ \\
$125$
& $54{,}000{,}000$
& $39.25\ (130)$
& $27.89\ (32)$
& $16.77\ (40)$
& $21.20\ (16)$ \\
$216$
& $93{,}312{,}000$
& $39.75\ (111)$
& $42.11\ (34)$
& $16.86\ (29)$
& $32.62\ (17)$ \\
\hline
\end{tabular}
\end{table}

Table~\ref{tab:weak_scaling_total_time_iterations_spe10} reports the total solution times for AlgMortar and AMG using relative residual tolerances of $10^{-11}$ and $10^{-6}$. For $\mathrm{rtol}=10^{-11}$, the AlgMortar solution time increases from $26.21$ s on $27$ MPI processes to $39.75$ s on $216$ processes, while the PCG iteration count decreases from $184$ to $111$. In comparison, the AMG time increases from $10.51$ s to $42.11$ s, making AlgMortar slightly faster for the largest problem. For $\mathrm{rtol}=10^{-6}$, the AlgMortar time increases from $9.96$ s to $16.86$ s, with the iteration count decreasing from $44$ to $29$, whereas the AMG time grows from $8.73$ s to $32.62$ s. Consequently, AlgMortar is faster from $64$ MPI processes onward, with its advantage becoming more pronounced at $216$ processes.
These results show that AlgMortar remains robust as the SPE10 problem size increases and exhibits better weak scalability than AMG, particularly for $\mathrm{rtol}=10^{-6}$. Relaxing the tolerance reduces its solution time more significantly, while AMG remains dominated by setup costs.

\subsection{Contrast robustness}
\label{sec:constrast-robust}
We again consider the SPE10 permeability field used as the base case in
Section~\ref{sec:scalability-spe10} and study the effect of increasing the
permeability contrast. 
Following \cite{jaramillo2022billion,carvalho2025preconditioner}, we introduce a
parameter $\theta>0$ and replace the isotropic permeability field
$K=\kappa I$ by
$K_\theta = \kappa^\theta I$.
Thus, if the original contrast is $c$
then the contrast of the modified field is $c^\theta$.

\begin{table}[htbp]
\footnotesize
\caption{AlgMortar and AMG performance under SPE10 permeability amplification
$K_\theta=\kappa^\theta I$ on the $300\times600\times300$ and
$360\times720\times360$ grids.}
\label{tab:spe10-contrast-theta}
\begin{center}
\begin{tabular}{|c|c|c|c|c|c|}
\hline
\multirow{3}{*}{$\theta$}
& \multirow{3}{*}{Contrast}
& \multicolumn{4}{c|}{Solution time in seconds (PCG iterations)} \\
\cline{3-6}
&
& \multicolumn{2}{c|}{$300\times600\times300$}
& \multicolumn{2}{c|}{$360\times720\times360$} \\
\cline{3-6}
&
& AlgMortar
& AMG
& AlgMortar
& AMG \\
\hline
$0.6$
& $2.94 \times 10^{4}$
& $12.40\ (23)$
& $19.41\ (16)$
& $14.99\ (24)$
& $30.09\ (16)$ \\

$0.8$
& $9.07 \times 10^{5}$
& $13.87\ (29)$
& $20.87\ (17)$
& $16.63\ (29)$
& $31.73\ (17)$ \\

$1.0$
& $2.80 \times 10^{7}$
& $16.77\ (40)$
& $21.20\ (16)$
& $16.86\ (29)$
& $32.61\ (17)$ \\

$1.2$
& $8.64 \times 10^{8}$
& $19.07\ (49)$
& $22.00\ (16)$
& $17.27\ (31)$
& $33.53\ (17)$ \\

$1.4$
& $2.67 \times 10^{10}$
& $16.81\ (39)$
& $23.13\ (17)$
& $18.58\ (33)$
& $33.87\ (16)$ \\
\hline
\end{tabular}
\end{center}
\end{table}

We adopt $\mathrm{rtol}=10^{-6}$ and consider grids of
$300\times600\times300$ ($54{,}000{,}000$ DOFs) on 125 MPI
processes and $360\times720\times360$ ($93{,}312{,}000$ DOFs) on
216 MPI processes. The results are reported in
Table~\ref{tab:spe10-contrast-theta}.

AlgMortar remains robust as the permeability contrast increases, although its
iteration count is more sensitive to $\theta$ than that of AMG. On the smaller
grid, its iteration count increases from 23 to 49 for
$0.6\leq\theta\leq1.2$, then decreases to 39 at $\theta=1.4$; the time
similarly rises from 12.40 s to 19.07 s and drops to 16.81 s. AMG iterations
remain nearly constant, while time increases from 19.41 s to 23.13 s. On the
larger grid, AlgMortar time increases from 14.99 s to 18.58 s, compared with
30.09 to 33.87 s for AMG. Thus, AMG is more stable in iteration count, but
AlgMortar is faster for every contrast and both grids, especially the larger
one.

\subsection{SPE11 corner-point grid benchmark}
\label{sec:spe11}
To evaluate the proposed preconditioner on more challenging grids, we consider
the SPE11 benchmark \cite{nordbotten2024spe11}. One of its test cases, SPE11B,
is a two-dimensional field-scale geological model in the $x$-$z$ plane.
The model can be discretized using a corner-point grid, a representation commonly employed in reservoir modeling that may include cells with extreme aspect ratios, nonorthogonal geometries, inactive cells, and nearly collapsed cells (Fig.~\ref{fig:spe11b-permeability-xz-slice}).

\begin{figure}[!ht]
    \centering
    \includegraphics[width=1\linewidth]{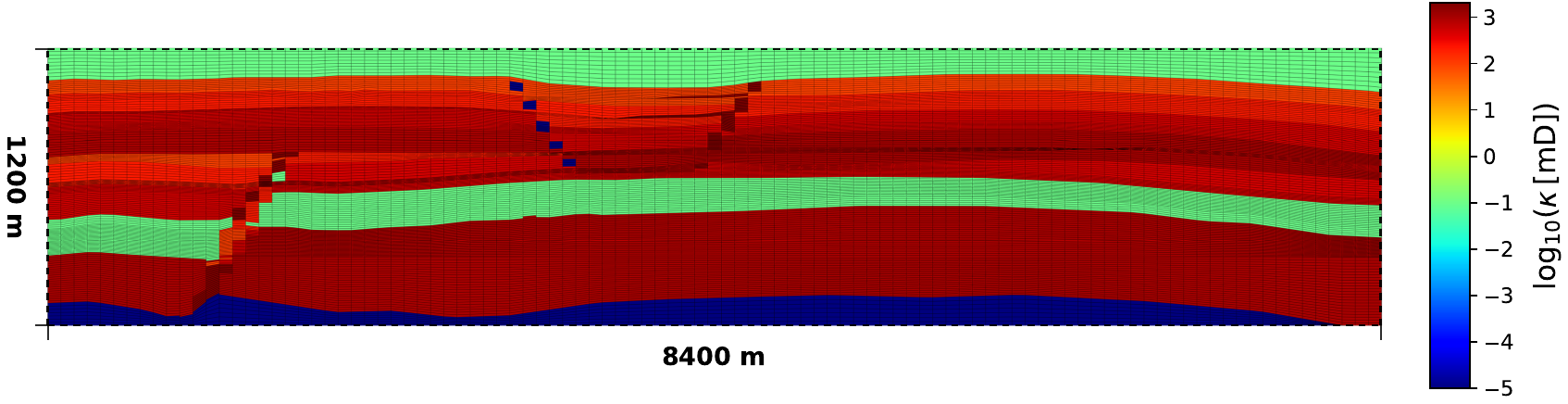}
\caption{Permeability field $\kappa$ and corner point grid generated with
\texttt{pyopmspe11} for the SPE11B model. Colors represent
$\log_{10}(\kappa)$, with $\kappa$ measured in millidarcies (mD).}
    \label{fig:spe11b-permeability-xz-slice}
\end{figure}

The SPE11C case extends SPE11B to three dimensions by extrusion in the
$y$-direction followed by a parabolic deformation, as shown in
Fig.~\ref{fig:spe11c-permeability-3d-1m}. The corner-point grids and geological
properties are generated with the open-source \texttt{pyopmspe11} framework
\cite{marban2025pyopmspe11}. The domain spans
$8400\,\mathrm{m}\times5000\,\mathrm{m}\times1200\,\mathrm{m}$, and larger grids
are obtained by refining the $x$- and $y$-directions while preserving the
vertical resolution and geological structure. A small number of logical cells
are excluded during grid processing. The resulting files are processed with a
custom implementation based on OPM \cite{rasmussen2021opm}, which discretizes
the Darcy flow problem using TPFA and assembles the system matrix and
right-hand-side vector.
\begin{figure}[!ht]
    \centering
    \includegraphics[width=1\linewidth]{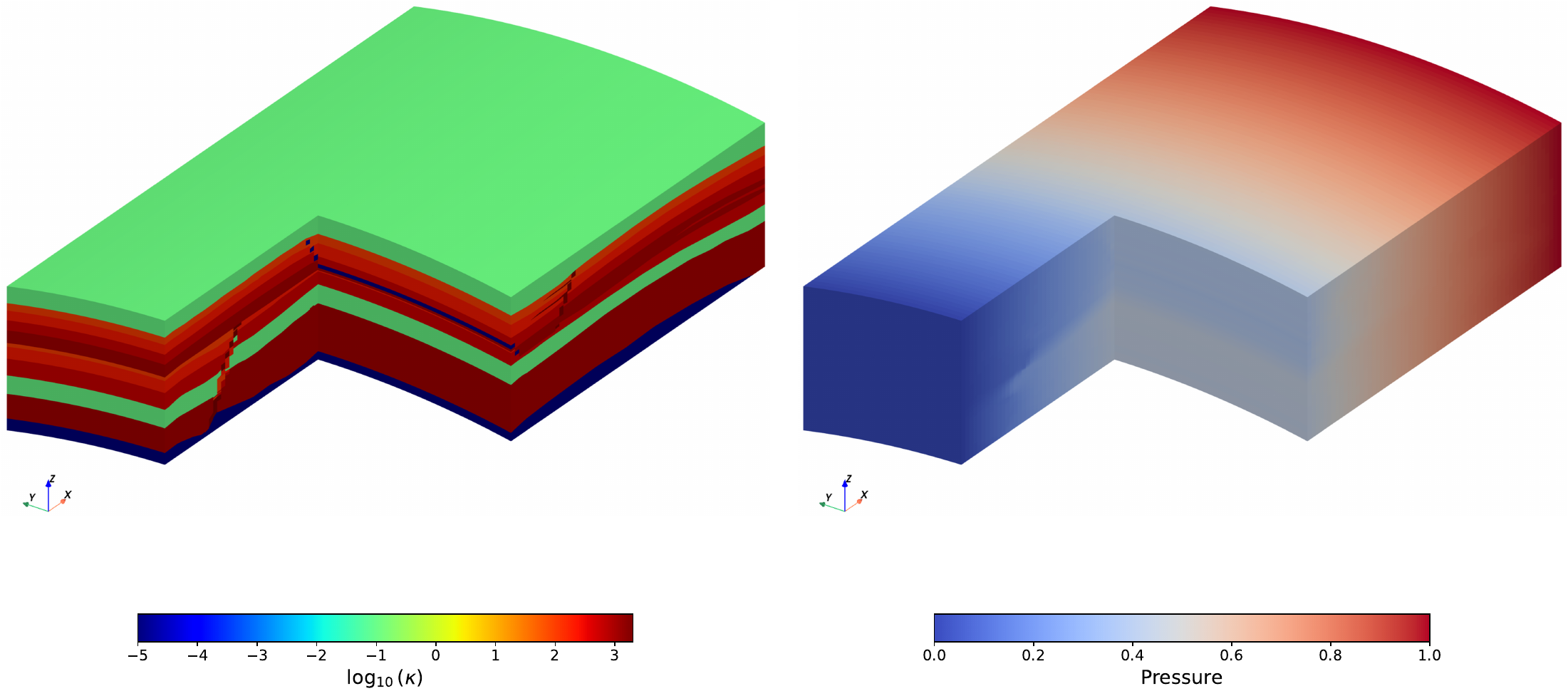}
\caption{Cutaway views of the SPE11C permeability field (left) and
pressure field (right) on the $103\times62\times157$ corner-point grid, which
contains $1{,}001{,}300$ active cells. Permeability colors represent
$\log_{10}(\kappa)$, with $\kappa$ in mD.}
    \label{fig:spe11c-permeability-3d-1m}
\end{figure}

The permeability field $\kappa$ generated for our experiments is shown in
Fig.~\ref{fig:spe11b-permeability-xz-slice} for the two-dimensional case and in
Fig.~\ref{fig:spe11c-permeability-3d-1m} for the three-dimensional case.
The permeability tensor is defined by
$K({x})=\operatorname{diag}\!\left(\kappa({x}),
\kappa({x}),0.1\kappa({x})\right)$, introducing vertical
anisotropy. Dirichlet conditions $p=0$ and $p=1$ are imposed at
$x=0$ and $x=8400\,\mathrm{m}$, respectively, while homogeneous Neumann
conditions are prescribed elsewhere. The resulting pressure field is shown in
Fig.~\ref{fig:spe11c-permeability-3d-1m}.

\begin{table}[htbp]
\footnotesize
\caption{Solver performance for the SPE11C cases 50M and 100M.}
\label{tab:timings_spe11}
\begin{center}
\begin{tabular}{|c|c|c|c|c|c|}
\hline
\multirow{3}{*}{Case}
& \multirow{3}{*}{\# MPI proc.}
& \multicolumn{4}{c|}{Solution time in seconds (PCG iterations)} \\
\cline{3-6}
&
& \multicolumn{2}{c|}{$\mathrm{rtol}=10^{-11}$}
& \multicolumn{2}{c|}{$\mathrm{rtol}=10^{-6}$} \\
\cline{3-6}
&
& AlgMortar
& AMG
& AlgMortar
& AMG \\
\hline
$50\,\mathrm{M}$
& $125$
& $25.73\ (83)$
& $29.72\ (29)$
& $11.80\ (19)$
& $23.50\ (13)$ \\

$100\,\mathrm{M}$
& $250$
& $38.74\ (108)$
& $34.84\ (32)$
& $16.23\ (23)$
& $25.95\ (15)$ \\
\hline
\end{tabular}
\end{center}
\end{table}

We consider two large-scale SPE11C cases with approximately
$50$ million ($50$M) and $100$ million ($100$M) cells.
The $50$M grid ($731\times436\times157$) contains $50{.}038{.}412$
cells, of which $49{.}974{.}320$ are active. The $100$M grid
($1035\times617\times157$) contains $100{.}259{.}415$ cells, of which
$100{.}129{.}845$ are active.

The $50$M and $100$M cases use $125$ and $250$ MPI processes, respectively, and are compared with AMG using $\mathrm{rtol}=10^{-6}$ and $\mathrm{rtol}=10^{-11}$.
Table~\ref{tab:timings_spe11} shows that AlgMortar is faster for both tolerances in the $50$M case and substantially faster for $\mathrm{rtol}=10^{-6}$ in both cases. AMG is slightly faster only for the $100$M case with $\mathrm{rtol}=10^{-11}$. These results demonstrate that AlgMortar remains competitive with AMG even on challenging corner-point grids.

\subsection{SuiteSparse matrices}
\label{sec:suitesparse}
We conclude our numerical experiments by considering matrices arising from applications beyond the TPFA discretizations considered previously. The goal is to assess whether AlgMortar can also provide good performance for other classes of SPD systems. To this end, we select seven SPD matrices from the SuiteSparse Matrix Collection\footnote{\url{https://sparse.tamu.edu/}}, formerly known as the University of Florida Sparse Matrix Collection \cite{david2011sparse}. The main characteristics of the selected matrices are summarized in Table~\ref{tab:suitesparse_matrices}.

\begin{table}[!ht]
\centering
\footnotesize
\caption{Characteristics of the SuiteSparse matrices used in this work.}
\label{tab:suitesparse_matrices}
\renewcommand{\arraystretch}{1.2}
\setlength{\tabcolsep}{5pt}

\begin{tabularx}{\textwidth}{|l|r|r|>{\raggedright\arraybackslash}X|}
\hline
Matrix & DOFs & \# Nonzeros & Description \\
\hline

\texttt{Geo\_1438}
& $1{,}437{,}960$
& $60{,}236{,}322$
& 3D geomechanical problem; tetrahedral FEM with three displacement unknowns per node. \\
\hline

\texttt{Emilia\_923}
& $923{,}136$
& $40{,}373{,}538$
& 3D geomechanical problem; tetrahedral FEM on an irregular mesh. \\
\hline

\texttt{Flan\_1565}
& $1{,}564{,}794$
& $114{,}165{,}372$
& 3D structural problem for a steel flange; hexahedral FEM on a structured mesh. \\
\hline

\texttt{Hook\_1498}
& $1{,}498{,}023$
& $59{,}374{,}451$
& 3D structural problem for a steel hook; regularly shaped tetrahedral elements. \\
\hline

\texttt{G3\_circuit}
& $1{,}585{,}478$
& $7{,}660{,}826$
& Circuit simulation problem. \\
\hline

\texttt{ecology2}
& $999{,}999$
& $4{,}995{,}991$
& Landscape ecology problem. \\
\hline

\texttt{thermal2}
& $1{,}228{,}045$
& $8{,}580{,}313$
& Unstructured FEM for a steady-state thermal problem.
\\\hline

\end{tabularx}
\end{table}
\begin{table}[!ht]
\centering
\footnotesize
\caption{Strong-scaling study for the matrices listed in Table~\ref{tab:suitesparse_matrices}. Times are reported in seconds, with the number of PCG iterations in parentheses. Parallel efficiency is computed relative to the 4-process case.}\label{tab:strong_scaling_suitesparse}

\begin{tabular}{|l|c|c|c|c|c|}
\hline
\multirow{2}{*}{Matrix}
& \multirow{2}{*}{\# MPI Proc.}
& \multicolumn{2}{c|}{AlgMortar}
& \multicolumn{2}{c|}{AMG} \\
\cline{3-6}
&
& Time (it.) & Efficiency (\%)
& Time (it.) & Efficiency (\%) \\
\hline

\multirow{4}{*}{\texttt{Geo\_1438}}
& 4  & 43.35 (143) & -      & 48.18 (155) & - \\
& 8  & 20.39 (160) & 106.32 & 24.99 (158) & 96.40 \\
& 16 & 11.73 (175) & 92.38  & 14.38 (161) & 83.74 \\
& 32 & 7.06 (184)  & 76.76  & 9.78 (170)  & 61.60 \\
\hline

\multirow{4}{*}{\texttt{Emilia\_923}}
& 4  & 96.27 (746)  & -      & 128.59 (775) & - \\
& 8  & 45.46 (815)  & 105.88 & 70.58 (848)  & 91.09 \\
& 16 & 33.72 (1092) & 71.37  & 44.80 (903)  & 71.75 \\
& 32 & 21.93 (999)  & 54.88  & 31.49 (916)  & 51.03 \\
\hline

\multirow{4}{*}{\texttt{Flan\_1565}}
& 4  & 315.70 (889)  & -      & 566.91 (1107) & - \\
& 8  & 143.07 (958)  & 110.33 & 308.40 (1197) & 91.91 \\
& 16 & 83.25 (1258)  & 94.81  & 202.85 (1418) & 69.87 \\
& 32 & 43.92 (1133)  & 89.86  & 135.84 (1588) & 52.17 \\
\hline

\multirow{4}{*}{\texttt{Hook\_1498}}
& 4  & 67.16 (317) & -     & 242.32 (693) & - \\
& 8  & 36.77 (398) & 91.33 & 125.91 (696) & 96.23 \\
& 16 & 23.35 (509) & 71.90 & 63.02 (629)  & 96.13 \\
& 32 & 16.49 (547) & 50.92 & 44.05 (722)  & 68.76 \\
\hline

\multirow{4}{*}{\texttt{G3\_circuit}}
& 4  & 5.37 (51) & -     & 1.70 (15) & - \\
& 8  & 2.89 (52) & 92.90 & 0.92 (15) & 92.56 \\
& 16 & 1.47 (46) & 91.25 & 0.55 (16) & 77.83 \\
& 32 & 0.88 (44) & 76.70 & 0.35 (15) & 61.24 \\
\hline

\multirow{4}{*}{\texttt{ecology2}}
& 4  & 2.56 (38) & -     & 0.91 (22) & - \\
& 8  & 1.34 (35) & 95.20 & 0.47 (21) & 97.84 \\
& 16 & 0.83 (35) & 77.33 & 0.27 (21) & 84.30 \\
& 32 & 0.50 (29) & 64.12 & 0.19 (21) & 60.39 \\
\hline

\multirow{4}{*}{\texttt{thermal2}}
& 4  & 3.81 (36) & -      & 1.43 (21) & - \\
& 8  & 1.87 (32) & 102.15 & 0.75 (22) & 95.34 \\
& 16 & 0.94 (26) & 101.33 & 0.41 (22) & 87.55 \\
& 32 & 0.55 (22) & 86.85  & 0.24 (22) & 73.46 \\
\hline
\end{tabular}
\end{table}

Table~\ref{tab:strong_scaling_suitesparse} reports the solution times and strong-scaling behavior of AlgMortar and AMG. These experiments use a relative tolerance of $10^{-6}$, with AlgMortar employing 50 subdomains and 10 smoothing blocks per MPI process. The matrices were reordered using the Reverse Cuthill-McKee algorithm, as also done in \cite{janna2011blockFSAI}. AlgMortar is consistently faster for \texttt{Flan\_1565}, \texttt{Hook\_1498}, \texttt{Geo\_1438}, and \texttt{Emilia\_923}. In particular, for \texttt{Flan\_1565}, AlgMortar maintains a parallel efficiency of $89.86\%$ at 32 processes, compared with $52.17\%$ for AMG.

For \texttt{G3\_circuit}, \texttt{ecology2}, and \texttt{thermal2}, AMG yields lower solution times at all process counts. It is worth noting that these matrices are considerably sparser than the previous cases, suggesting that the relative performance of the two preconditioners is also influenced by the matrix sparsity. 
However, AlgMortar exhibits better strong-scaling behavior for \texttt{G3\_circuit} and \texttt{thermal2}, while both methods show comparable scaling for \texttt{ecology2}. Overall, the results show that AlgMortar can also be effectively applied to matrices beyond the TPFA setting for which it was originally designed, outperforming AMG in some cases.

\section{Conclusion}
\label{sec:conclusions}

This paper presented AlgMortar, a fully algebraic extension of the Multiscale Mortar Mixed Finite Element Method (MMMFEM). For SPD matrices, we established well-posedness, positive definiteness of the preconditioner, and condition-number bounds related to the approximation properties of the interface space.
Numerical experiments showed good scalability for single-phase Darcy-flow problems discretized by TPFA, with performance competitive with AMG, as well as competitive results for matrices arising from other applications and discretizations.
Future work will address transient multiphase problems, improved interface spaces and scalable interface solvers, and alternative graph-partitioning strategies.

\bibliographystyle{siamplain}
\bibliography{references}
\end{document}